\documentclass[11pt]{amsart}

\usepackage{amsfonts,amsmath,amssymb,amsthm,mathabx,shuffle,latexsym,xcolor,tikz-cd}
\usepackage{enumitem}

\newtheorem{theorem}{Theorem}[section]
\newtheorem{lemma}[theorem]{Lemma}
\newtheorem{proposition}[theorem]{Proposition}
\newtheorem{corollary}[theorem]{Corollary}

\newtheorem*{thma}{Theorem A}
\newtheorem*{thmb}{Theorem B}
\newtheorem*{thmc}{Theorem C}
\newtheorem*{thmd}{Theorem D}

\theoremstyle{definition}
\newtheorem{definition}[theorem]{Definition}
\newtheorem{example}[theorem]{Example}

\theoremstyle{remark}
\newtheorem{remark}[theorem]{Remark}

\numberwithin{equation}{section}

\newcommand{\qKT}{\ensuremath{\mathrm{qKT}}}
\newcommand{\QqKT}{\ensuremath{\mathrm{QqKT}}}
\newcommand{\HqKT}{\ensuremath{\mathrm{HqKT}}}
\newcommand{\LRS}{\ensuremath{\mathrm{LRS}}}
\newcommand{\HLRS}{\ensuremath{\mathrm{HLRS}}}
\newcommand{\KeyTab}{\ensuremath{\mathrm{KeyTab}}}

\newcommand{\SSF}{\ensuremath{\mathrm{SSF}}}
\newcommand{\ASSF}{\atom\ensuremath{\mathrm{SSF}}}
\newcommand{\MSSF}{\Mono\ensuremath{\mathrm{SSF}}}
\newcommand{\FSSF}{\Fund\ensuremath{\mathrm{SSF}}}
\newcommand{\LSSF}{\Part\ensuremath{\mathrm{SSF}}}
\newcommand{\HSSF}{\ensuremath{\mathrm{HSSF}}}

\newcommand{\SSAF}{\ensuremath{\mathrm{SSAF}}}
\newcommand{\QSSF}{\ensuremath{\mathrm{QSSF}}}
\newcommand{\revSSYT}{\ensuremath{\mathrm{RT}}}
\newcommand{\Pairs}{\ensuremath{\mathrm{Pairs}}}
\newcommand{\HPairs}{\ensuremath{\mathrm{HPairs}}}
\newcommand{\HrevSSYT}{\ensuremath{\mathrm{HRT}}}
\newcommand{\QrevSSYT}{\ensuremath{\mathrm{QRT}}}

\newcommand{\wt}{\ensuremath{\mathrm{wt}}}
\newcommand{\destand}{\ensuremath{\mathrm{dst}}}

\newcommand{\sch}{\ensuremath{\mathfrak{S}}}
\newcommand{\key}{\ensuremath{\kappa}}
\newcommand{\qkey}{\ensuremath{\mathfrak{Q}}}
\newcommand{\Fund}{\ensuremath{\mathfrak{F}}}
\newcommand{\Mono}{\ensuremath{\mathfrak{M}}}
\newcommand{\Part}{\ensuremath{\mathfrak{L}}}
\newcommand{\atom}{\ensuremath{\mathcal{A}}}

\newcommand{\flatten}{\ensuremath{\mathrm{flat}}}
\newcommand{\sort}{\ensuremath{\mathrm{sort}}}

\newcommand{\swap}{\ensuremath{\mathrm{swap}}}
\newcommand{\Shuffle}{\ensuremath{\mathrm{Shuffle}}}

\newcommand{\lswap}{\ensuremath{\mathrm{lswap}}}
\newcommand{\Qlswap}{\ensuremath{\mathrm{Qlswap}}}

\newcommand{\sgeq}{\ensuremath{\unrhd}}

\newcommand{\excise}[1]{}

\newcommand{\FS}{\ensuremath{\mathrm{FixSlide}}}
\newcommand{\Slide}{\ensuremath{\mathrm{Slide}}}

\newlength\cellsize 

\savebox2{%
\begin{picture}(12,12)
\put(0,0){\line(1,0){12}}
\put(0,0){\line(0,1){12}}
\put(12,0){\line(0,1){12}}
\put(0,12){\line(1,0){12}}
\end{picture}}

\newcommand\cellify[1]{\def\thearg{#1}\def\nothing{}%
\ifx\thearg\nothing\vrule width0pt height\cellsize depth0pt%
  \else\hbox to 0pt{\usebox2\hss}\fi%
  \vbox to 12\unitlength{\vss\hbox to 12\unitlength{\hss$#1$\hss}\vss}}

\newcommand\tableau[1]{\vtop{\let\\=\cr
\setlength\baselineskip{-12000pt}
\setlength\lineskiplimit{12000pt}
\setlength\lineskip{0pt}
\halign{&\cellify{##}\cr#1\crcr}}}

\begin{document}


\title[Polynomial bases]{Polynomial bases: positivity and Schur multiplication}  

\author[D. Searles]{Dominic Searles}
\address{Department of Mathematics and Statistics, University of Otago, Dunedin 9016, New Zealand}
\email{dominic.searles@otago.ac.nz}




\date{July 30, 2018}


\begin{abstract}
We establish a poset structure on combinatorial bases of multivariate polynomials defined by positive expansions, and study properties common to bases in this poset. Included are the well-studied bases of Schubert polynomials, Demazure characters and Demazure atoms; the quasi-key, fundamental and monomial slide bases introduced in 2017 by Assaf and the author; and a new basis we introduce completing this poset structure. We show the product of a Schur polynomial and an element of a basis in this poset expands positively in that basis; in particular, we give the first Littlewood-Richardson rule for the product of a Schur polynomial and a quasi-key polynomial. This rule simultaneously extends Haglund, Luoto, Mason and van Willigenburg's (2011) Littlewood-Richardson rule for quasi-Schur polynomials and refines their Littlewood-Richardson rule for Demazure characters. We also establish bijections connecting combinatorial models for these polynomials including semi-skyline fillings and quasi-key tableaux.
\end{abstract}

\maketitle

\tableofcontents

%
\section{Introduction}\label{sec:intro}

%

The ring of multivariate polynomials has certain bases, indexed by weak compositions $a$, which enjoy important applications to geometry and representation theory. The \emph{Schubert polynomials} $\{\sch_w\}$ introduced by Lascoux and Sch\"utzenberger \cite{LS82} are a principal example. These polynomials, typically indexed by permutations $w$ (which are in bijection with weak compositions), represent Schubert classes in the cohomology of the complete flag variety.  Another example is the basis of \emph{Demazure characters} $\{\key_a\}$, also known as \emph{key polynomials}, introduced by Demazure \cite{Dem74} and studied combinatorially in type A by Lascoux and Sch\"utzenberger \cite{LS90} and Reiner and Shimozono \cite{RS95}. These polynomials are characters of Demazure modules, and Ion \cite{Ion} proved they occur as specializations of the nonsymmetric Macdonald polynomials introduced in \cite{Opdam}, \cite{Macdonald} and \cite{Cherednik}.

The \emph{Demazure atoms} $\{\atom_a\}$, introduced by Lascoux and Sch\"utzenberger \cite{LS90} and originally called \emph{standard bases}, form another basis of the polynomial ring. Schubert polynomials and Demazure characters expand positively in Demazure atoms \cite{LS90}, and Mason  \cite{Mason} showed Demazure atoms are also realized as specializations of nonsymmetric Macdonald polynomials. The Demazure atom basis has been widely studied from a combinatorial perspective \cite{HHL08} \cite{Mason08}, \cite{Mason}, \cite{HLMvW11b}, \cite{Ale16}, \cite{Pun16}, \cite{Mon16}. Demazure atoms were used by Haglund, Luoto, Mason and van Willigenburg \cite{HLMvW11a} to introduce the \emph{quasi-Schur} basis of quasisymmetric polynomials, which has applications to e.g., representation theory of Hecke algebras \cite{TvW15}.

Assaf and the author \cite{AS1}, \cite{AS2} introduced three bases of the polynomial ring, with the goal of better understanding Schubert polynomials and Demazure characters. The \emph{fundamental slide} basis $\{\Fund_a\}$ of polynomials lifts the fundamental basis \cite{Ges84} of quasisymmetric polynomials, in the sense that every fundamental quasisymmetric polynomial is also a fundamental slide polynomial, and the stable limit of a fundamental slide polynomial is a fundamental quasisymmetric function. A fundamental slide polynomial may also be realized as the weighted sum over the \emph{compatible sequences} of \cite{BJS93} for a reduced word for an element of the symmetric group. Similarly the \emph{monomial slide} basis $\{\Mono_a\}$ lifts the monomial basis, and the \emph{quasi-key} basis $\{\qkey_a\}$ lifts the quasi-Schur basis \cite{HLMvW11a} of quasisymmetric polynomials. Schubert polynomials and Demazure characters also expand positively in all three of these bases.

The fundamental slide basis has found application in several different areas. Examples include a lifting of the multi-fundamental quasisymmetric functions of Lam and Pylyavskyy \cite{LP07} resulting in a new formula \cite{PS17} for Grothendieck polynomials (which represent Schubert classes in the $K$-theory ring of the complete flag variety); a generalization of dual equivalence \cite{Assaf-weak}; a new proof of Kohnert's rule for Schubert polynomials \cite{Assaf-models}; and a proof that the $t=0$ specialization of a nonsymmetric Macdonald polynomial expands positively in Demazure characters \cite{Assaf-nonsymmetric}. 

In this paper we establish a poset $\mathcal{P}$ explaining when one basis expands positively in one another, give formulas for these expansions, extract properties common to bases in this poset, and provide bijections between combinatorial models used to define these bases. The Hasse diagram of $\mathcal{P}$ is shown in Figure~\ref{fig:expand}, where for $\mathcal{B}, \mathcal{B'}\in \mathcal{P}$, we have $\mathcal{B}>\mathcal{B'}$ if there is a directed path from $\mathcal{B}$ to $\mathcal{B'}$. Along the way, we are led to introduce a new basis $\{\Part_a\}$ which, in keeping with the atom and fundamental slide terminology, we call the \emph{fundamental particle} basis.

\begin{figure}[ht]
\[
\begin{tikzcd}
\sch_w \arrow[rr] & & \key_a \arrow[rr] & & \qkey_a  \arrow[rr] \arrow[d] & & \Fund_a  \arrow[rr] \arrow[d] & & \Mono_a \arrow[d] &  \\
& & & & \atom_a  \arrow [rr] & & \Part_a \arrow[rr] & & x^a  &
\end{tikzcd}
\]
  \caption{\label{fig:expand}The positivity poset $\mathcal{P}$ on combinatorial bases of polynomials.}
\end{figure}

\begin{thma}\label{thm:positivity}
Given the poset $\mathcal{P}$ on polynomial bases whose Hasse diagram is shown in Figure~\ref{fig:expand}, for $\mathcal{B}$ a basis in $\mathcal{P}$, all $f\in \mathcal{B}$ expand positively in $\mathcal{B}'\in \mathcal{P}$ if and only if $\mathcal{B} > \mathcal{B}'$ in $\mathcal{P}$. 
\end{thma}

Positivity of the expansion of Schubert polynomials into Demazure characters was proved by Lascoux and Sch\"utzenberger \cite{LS90} and Reiner and Shimozono \cite{RS95}. Positivity of the expansion of Demazure characters into Demazure atoms was proved in \cite{LS90} and also by Mason in \cite{Mason08}. Positivity of the other expansions in the top row of Figure~\ref{fig:expand} were proved by Assaf and the author \cite{AS1}, \cite{AS2}. In this paper, we establish the remaining positivity relationships in $\mathcal{P}$, in particular, we give positive combinatorial formulas for the expansion of quasi-key polynomials in Demazure atoms and for the expansions of both Demazure atoms and fundamental slide polynomials in the new basis of fundamental particles, and we complete the proof of Theorem A by showing that for each pair of bases incomparable in $\mathcal{P}$, neither expands positively in the other.

We continue by establishing properties common to bases in $\mathcal{P}$. A main result is that every basis $\mathcal{B}\in \mathcal{P}$ satisfies a positivity property for multiplication with Schur polynomials $s_\lambda$.

\begin{thmb}\label{thm:LR}
For any polynomial basis $\mathcal{B}$ in the poset $\mathcal{P}$, any weak composition $a$ of length $n$ and $f_a\in \mathcal{B}$ the polynomial indexed by $a$, the product 
\[f_a \cdot s_{\lambda}(x_1,\ldots x_n)\] 
expands positively in the basis $\mathcal{B}$.
\end{thmb}

For Schubert polynomials, this statement is clear from the associated geometry: the structure constants are counting points in the intersection of three Schubert subvarieties in general position, two for the complete flag variety and one for a Grassmannian variety. A certain case of this statement is also proved combinatorially in \cite{KY04}. For fundamental and monomial slide polynomials (also ordinary monomials, trivially) this statement is clear from the fact these bases have positive structure constants and the fact that Schur polynomials expand positively in these bases, proved in \cite{AS1}. The remaining four bases do not have positive structure constants. For Demazure characters and Demazure atoms, Littlewood-Richardson rules for products with Schur polynomials are given by Haglund, Luoto, Mason and van Willigenburg in \cite{HLMvW11b}. In this paper, we will provide Littlewood-Richardson rules for products with Schur polynomials for the two remaining bases: quasi-key polynomials and fundamental particles.

The bases in Figure~\ref{fig:expand} may be described in terms of the combinatorics of \emph{semi-skyline fillings}. These were introduced in \cite{HHL08}, where they were referred to as fillings of column diagrams. The \emph{skyline diagram} $D(a)$ of a weak composition $a$ is the diagram with $a_i$ boxes in row $i$, left-justified. In our convention, which is upside-down from that of \cite{HLMvW11b} and rotated $90$ degrees counterclockwise from that of \cite{Mason08}, the lowest row is row 1. A semi-skyline filling is an assignment of positive integers to the boxes of $D(a)$, along with potentially an extra positive integer assigned to the left of each row called a \emph{basement entry}, satisfying:

\begin{itemize}
\item Entries weakly decrease along rows
\item Entries in any single column are distinct
\item All \emph{triples} of entries are \emph{inversion} triples (see Figure~\ref{fig:triple}).
\end{itemize}

Denote the set of semi-skyline fillings of $D(a)$ by $\SSF(a)$.

\begin{thmc}\label{thm:skyline}
For any polynomial basis $\mathcal{B}$ in the poset $\mathcal{P}$ and any weak composition $a$, the polynomial $f_a\in \mathcal{B}$ can be described in terms of semi-skyline fillings. In particular, if $\mathcal{B}\neq\{\sch_w\}$ then there is a subset $\mathcal{B}\SSF(a)$ of $\SSF(a)$ such that 
\[f_a = \sum_{S \in \mathcal{B}\SSF(a)}x^{\wt(S)}\]
where $\wt(S)$ is the weak composition whose $i$th part is the number of entries $i$ appearing in $S$.
\end{thmc}

A Schubert polynomial $\sch_w$ cannot in general be described in terms of semi-skyline fillings for the \emph{single} weak composition associated to $w$. 
However, since Schubert polynomials expand positively in Demazure characters, $\sch_w$ can be described in terms of semi-skyline fillings for the weak compositions involved in this expansion.

Demazure characters (in $n$ variables) are described in this manner by imposing the basement whose $i$th entry is $n-i+1$ and extending the weakly decreasing and triple conditions to include basement entries (see \cite{HLMvW11b}). For the other three bases in the top row, rather than using a basement we impose the condition that entries in the first column are at most their row index, and decrease from top to bottom. With no further conditions, such semi-skyline fillings define quasi-key polynomials. For fundamental slide polynomials, we impose the additional condition that if a box $B$ is in a higher row that a box $B'$, then the entry in box $B$ is strictly larger than the entry in box $B'$. For monomial slide polynomials, we impose the additional condition that all boxes in the same row have the same entry.
For the bases in the bottom row of $\mathcal{P}$, the associated conditions on semi-skyline fillings are obtained from those of the basis above (in Figure~\ref{fig:expand}) by replacing ``entries in the first column decrease from top to bottom'' with ``entries in the first column are equal to their row index''.

The skyline filling description for Demazure atoms is given by Mason \cite{Mason08}. Instead of stating that entries in the first column are equal to their row index, this rule for Demazure atoms uses the basement whose $i$th entry is $i$ and extends the triple conditions to the basement, which has the same effect. The skyline descriptions for fundamental and monomial slide polynomials and fundamental particles are new here, though they follow straightforwardly from the definitions of these polynomials. The skyline description for quasi-key polynomials is also new here, but is not clear from the definition and will follow from theorems proved in Section 3.

For a basis $\mathcal{B}\in \mathcal{P}$, one can also ask about the \emph{stable limit} $\lim_{m\to \infty} f_{0^m\times a}$ of a polynomial $f_a\in \mathcal{B}$, where $0^m\times a$ is the weak composition obtained by prepending $m$ zeros to $a$.

\begin{thmd}\label{thm:stable}
Stable limits exist for all the bases in the top row of $\mathcal{P}$, and for none of the bases in the bottom row.
\end{thmd}
Schubert polynomials stabilize to Stanley symmetric functions \cite{Mac91}. Demazure characters stabilize to Schur functions (implicit in work of Lascoux and Sch\"utzenberger \cite{LS90}, made explicit in \cite{AS2}). Quasi-key polynomials stabilize \cite{AS2} to the \emph{quasi-Schur} functions of \cite{HLMvW11a}, while fundamental and monomial slide polynomials stabilize \cite{AS1} to the fundamental and monomial quasisymmetric functions of \cite{Ges84}.

On the other hand, it is immediate from the rules for the skyline fillings that the bases $\mathcal{B}$ in the bottom row all have the property that if $a_i>0$ then $g_a \in \mathcal{B}$ is divisible by $x_i$. It follows that for any weak composition $a$ and any number $n$, 
\[g_{0^m \times a}(x_1,\ldots , x_n, 0,0, \ldots , 0) = 0\]
for all sufficiently large $m$, and therefore $\lim_{m\to \infty} g_{0^m\times a} = 0$.

In \cite{AS2}, quasi-key polynomials are defined combinatorially in terms of \emph{quasi-Kohnert} (or \emph{quasi-key}) tableaux, based on an algorithm of Kohnert \cite{Koh91} whereas in \cite{Mason08}, Demazure atoms are defined combinatorially in terms of semi-skyline fillings. These two tableau models are both defined on the skyline diagram of a weak composition, but have quite different rules for the fillings allowed. Our proof that quasi-key polynomials can also be realized as generating functions of semi-skyline fillings allows us to set up a dictionary between these tableau models. In particular, we give a bijection between semi-skyline fillings and quasi-key tableaux with fixed first column, passing through reverse tableaux. As a consequence, we show that quasi-key tableaux and semi-skyline fillings are constructed by decomposing reverse tableaux into sets of runs in an essentially dual manner: the former selecting increasing runs right-to-left, and the latter selecting decreasing runs left-to-right. We prove moreover that this bijection preserves the underlying skyline diagram. We use the combinatorics of the fundamental particle basis to restrict this bijection to distinguished subsets of quasi-key tableaux and semi-skyline fillings, and restrict this yet further using the combinatorics of fundamental slide polynomials.

\subsection{Organization}

In Section 2, we review Demazure atoms, fundamental and monomial slide polynomials and quasi-key polynomials, and provide semi-skyline filling descriptions for fundamental and monomial slide polynomials. In Section 3 we state and prove an explicit positive formula for the Demazure atom expansion of a quasi-key polynomial, and state and prove a positive combinatorial Littlewood-Richardson rule for the quasi-key expansion of the product of a quasi-key polynomial and a Schur polynomial. In Section 4 we introduce the new basis of fundamental particles, give a semi-skyline description, and state and prove positive combinatorial formulas for the expansion of both a Demazure atom and a fundamental slide polynomial into this basis. We then prove a positive combinatorial Littlewood-Richardson rule for the fundamental particle expansion of the product of a fundamental particle and a Schur polynomial. In Section 5 we establish bijections between semi-skyline fillings and quasi-key tableaux.

\section{Definitions and preliminaries}\label{sec:definitions}
Here we review bases in $\mathcal{P}$ and some of their associated combinatorial models, and explain how these bases can be understood in terms of semi-skyline fillings. We delay introducing the new fundamental particle basis $\{\Part_a\}$ until Section~\ref{sec:particle}.
\subsection{Demazure atoms, Demazure characters and semi-skyline fillings}

A \emph{triple} of a skyline diagram is a collection of three boxes with two adjacent in a row and either (Type A) the third box is above the right box and the lower row is weakly longer, or (Type B) the third box is below the left box and the higher row is strictly longer. 

\begin{figure}[ht]
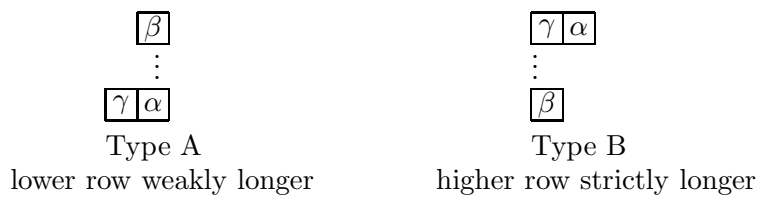

  \begin{displaymath}
    \begin{array}{l}
        \tableau{ & \beta } \\[-0.5\cellsize]  \hspace{1.5\cellsize} \vdots  \hspace{0.4\cellsize} \\ \tableau{  \gamma & \alpha } \\  \mbox{Type A} \\ \hspace{-3\cellsize} \mbox{lower row weakly longer}
    \end{array}
    \hspace{3\cellsize}
    \begin{array}{l}
   \hspace{3\cellsize}   \tableau{ \gamma & \alpha } \\[-0.5\cellsize] \hspace{3\cellsize}  \vdots \hspace{1.5\cellsize} \\ \hspace{3\cellsize} \tableau{ \beta & } \\ \hspace{3\cellsize}\mbox{Type B} \\  \mbox{higher row strictly longer}
    \end{array}
  \end{displaymath}
  \caption{\label{fig:triple}Triples for skyline diagrams.}
\end{figure}

Given a filling of the skyline diagram, a triple (of either type) is called an \emph{inversion triple} if either $\beta>\gamma \ge \alpha$ or $\gamma \ge \alpha > \beta$, and a \emph{coinversion triple} if $\gamma \ge \beta \ge \alpha$. 

Given a weak composition $a$, a \emph{semi-skyline filling} (or $\SSF$) of the skyline diagram $D(a)$ is a filling that weakly decreases along rows, has no repeated entries in any column, and all triples are inversion triples. Let $\SSF(a)$ denote the set of all semi-skyline fillings for $D(a)$.

The definition of a semi-skyline filling sometimes includes the notion of a \emph{basement}: an extra number assigned to each row, to the left of column 1, to which the weakly decreasing row, distinct column and triple conditions are extended. When a basement is involved, a semi-skyline filling is called a \emph{semi-skyline augmented filling} or $\SSAF$ \cite{Mason08}. Given a weak composition $a$,  let $\SSAF(a)$ denote the set of all semi-skyline augmented fillings of $D(a)$. Then the \emph{Demazure atom} $\atom_a$ is the polynomial given by the following combinatorial formula due to Mason:

\begin{theorem}[\cite{Mason08}]\label{thm:atomSSAF}
\[\atom_a = \sum_{S\in \SSAF(a)}x^{\wt(S)},\]
where the basement entry for row $i$ is $i$.
\end{theorem} 

The effect of this choice of basement is simply to force entries in the first column of each $\SSAF$ to be equal to their row index. Let $\ASSF(a)$ denote those $\SSF(a)$ whose first column entries are equal to their row index. Then, equivalently to Theorem~\ref{thm:atomSSAF},

\[\atom_a = \sum_{S\in \ASSF(a)}x^{\wt(S)}.\]

\begin{example}
\[\atom_{0103} = x^{0103} + x^{0112} +x^{0202} +x^{1102} +x^{0121} +x^{0211} +x^{1111},\]
which is computed by the elements of $\ASSF(0103)$ shown in Figure~\ref{fig:atom0103} below.
\end{example}

\begin{figure}[ht]
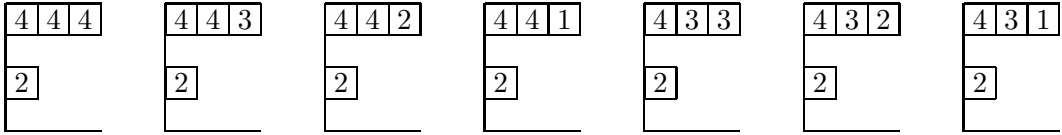

  \begin{center}
    \begin{displaymath}
      \begin{array}{c@{\hskip2\cellsize}c@{\hskip2\cellsize}c@{\hskip2\cellsize}c@{\hskip2\cellsize}c@{\hskip2\cellsize}c@{\hskip2\cellsize}c}
        \vline\tableau{  4 & 4 & 4 \\ \\  2 \\ \\\hline } &
        \vline\tableau{  4 & 4 & 3 \\ \\  2 \\ \\\hline } &
        \vline\tableau{  4 & 4 & 2 \\ \\  2 \\ \\\hline } &
        \vline\tableau{  4 & 4 & 1 \\ \\  2 \\ \\\hline } &
        \vline\tableau{  4 & 3 & 3 \\ \\  2 \\ \\\hline } &
        \vline\tableau{  4 & 3 & 2 \\ \\  2 \\ \\\hline } &
        \vline\tableau{  4 & 3 & 1 \\ \\  2 \\ \\\hline }
      \end{array}
    \end{displaymath}
    \caption{\label{fig:atom0103}The seven elements of $\ASSF(0103)$.}
  \end{center}
\end{figure}

In \cite{Dem74}, Demazure introduced the \emph{Demazure characters}, also known as \emph{key polynomials}. There are many different combinatorial formulas for computing the monomial expansion of a Demazure character \cite{RS95}. One formula, mentioned in the introduction, is a weighted sum of $\SSAF$ with a decreasing basement (\cite{Mason}). Another formula is a weighted sum of \emph{Kohnert tableaux} \cite{AS2}. Most useful for our purposes, however, will be the following expression of a Demazure character as a sum of Demazure atoms (see, e.g., \cite{Mason}, \cite{HLMvW11b}).

Given a weak composition $a$, let $\sort(a)$ be the rearrangement of the parts of $a$ into weakly decreasing order, and let $w$ be the minimal length permutation sending $a$ to $\sort(a)$. Then
\begin{equation}\label{eqn:keytoatom}
\key_a = \sum_{\stackrel{\sort(b)=\sort(a)}{v(b)\le w}} \atom_b
\end{equation}
where $v(b)$ is the minimal length permutation sending $b$ to $\sort(b)$ and $\le$ denotes the strong Bruhat order on permutations. 

\begin{example} Let $a=(0,1,0,3)$. Then $\sort(a)=(3,1,0,0)$, $w=3241$ and
\[\key_{0103} = \atom_{0103} + \atom_{1003} + \atom_{0130} + \atom_{1030} + \atom_{1300} + \atom_{0301} + \atom_{0310} + \atom_{3001} + \atom_{3010} + \atom_{3100}.\]
\end{example}

\subsection{Monomial and fundamental slide polynomials}

In \cite{AS1}, Assaf and the author define \emph{fundamental slide polynomials}. Given a weak composition $a$, let $\flatten(a)$ be the (strong) composition obtained by removing all $0$ terms from $a$. Given weak compositions $a,b$ of length $n$, we say $b$ \emph{dominates} $a$, denoted by $b \geq a$, if
\[
  b_1 + \cdots + b_i \geq a_1 + \cdots + a_i
\]
for all $i=1,\ldots,n$. Note this extends the usual dominance order on partitions.

\begin{definition}[\cite{AS1}]
  Given a weak composition $a$, the \emph{monomial slide polynomial} $\Mono_{a}$ is defined by
    \[\Mono_{a} = \sum_{\substack{b \geq a \\ \flatten(b) = \flatten(a)}} x^b.\]
\end{definition}

\begin{example}
\[  \Mono_{0103} = x^{0103} + x^{1003} + x^{0130} + x^{1030} + x^{1300}.\]
\end{example}

We can describe monomial slide polynomials in terms of semi-skyline fillings. Define a \emph{monomial semi-skyline filling} to be a semi-skyline filling whose entries in the first column are at most their row index and decrease from top to bottom, 
and all boxes in the same row have the same entry. Let $\MSSF(a)$ denote the set of all monomial semi-skyline fillings for $a$. Then the following is clear:

\begin{proposition}
\[\Mono_a = \sum_{S \in \MSSF(a)} x^{\wt(S)}.\]
\end{proposition}

\begin{definition}[\cite{AS1}]
  Given a weak composition $a$, the \emph{fundamental slide polynomial} $\Fund_{a}$ is defined by
  \[  \Fund_{a} = \sum_{\substack{b \geq a \\ \flatten(b) \ \mathrm{refines} \ \flatten(a)}} x^b.\]
\end{definition}

\begin{example}
\[  \Fund_{0103} = x^{0103} + x^{1003} + x^{0130} + x^{1030} + x^{1300} + x^{0112} + x^{1012} + x^{1102} + x^{1120} + x^{0121} + x^{1021} + x^{1201} + x^{1210} + x^{1111}.\]
\end{example}

We can also describe fundamental slide polynomials in terms of semi-skyline fillings. Define a \emph{fundamental semi-skyline filling} to be a semi-skyline filling whose entries in the first column are at most their row index, 
and if a box $B$ is in a higher row than a box $B'$, then the entry in box $B$ is strictly larger than the entry in box $B'$. Let $\FSSF(a)$ denote the set of all fundamental semi-skyline fillings for $a$. Then the following is clear:

\begin{proposition}\label{prop:fundamentalskyline}
\[\Fund_a = \sum_{S \in \FSSF(a)} x^{\wt(S)}.\]
\end{proposition}

\begin{remark}\label{rmk:fundtriple}
Notice that the additional conditions given for $\MSSF(a)$ and $\FSSF(a)$ both force all triples to be inversion triples, so the triple conditions are in fact
redundant for these families of semi-skyline fillings.
\end{remark}

Say that $b$ \emph{strongly dominates} $a$, denoted by $b \sgeq a$, if $b \geq a$ and for all $c \geq a$ such that $\flatten(c)=\flatten(b)$, we have $c \geq b$ as well. The fundamental slide polynomials expand positively in the monomial slide polynomials:
\begin{proposition}[\cite{AS1}]
  Given a weak composition $a$ of length $n$, we have
 \[   \Fund_{a} = \sum_{\substack{b \sgeq a \\ \flatten(b) \ \mathrm{refines} \ \flatten(a)}} \Mono_{b}.\]
\end{proposition}

\begin{example}
 \[ \Fund_{0103} = \Mono_{0103} + \Mono_{0112} + \Mono_{0121} + \Mono_{1111}.\]
\end{example}

\subsection{Quasi-key polynomials and quasi-key tableaux}

\begin{definition}\label{def:qktableaux}
Given a weak composition $a$, a \emph{quasi-key tableau of shape $a$} is a filling of $D(a)$ with positive integers satisfying
\begin{enumerate}
\item entries weakly decrease along rows, and no entry of row $i$ exceeds $i$
\item entries in any column are distinct, and entries increase up the first column
\item if an entry $i$ is above an entry $k$ in the same column with $i<k$, then there is a label $j$ immediately right of $k$, with $i<j$
\item Given two rows with the higher row strictly longer, then if entry $i$ is in column $c$ of the lower row and entry $j$ in column $c+1$ of the higher row, then  $i<j$.
\end{enumerate}
Denote the set of quasi-key tableaux of shape $a$ by $\qKT(a)$; for example, see Figure~\ref{fig:qKey0302}. Moreover, let $\qKT^{(1)}(a)$ be all quasi-key tableaux of shape $a$ whose first column entries are equal to their row index.
\end{definition}

\begin{remark}\label{rmk:quasi-key}
For quasi-key tableaux, boxes stay fixed and entries vary. The quasi-key conditions are a direct translation of the conditions of \cite{AS2} defining \emph{quasi-Kohnert tableaux}, for which entries stay fixed and boxes move. Given a quasi-key tableau, the corresponding quasi-Kohnert tableau is obtained by moving each box downward to the row given by its label, and changing its label to the index of the row it originally came from. 
\end{remark}

\begin{remark}
Quasi-key tableaux do not necessarily obey the triple conditions, so they are not semi-skyline fillings as we have defined them. However, we will show in Section 3 that the polynomial generated by $\qKT^{(1)}(a)$ is in fact the Demazure atom $\atom_a$. This will permit a new definition of quasi-key polynomials in terms of semi-skyline fillings.
\end{remark}

\begin{figure}[ht]
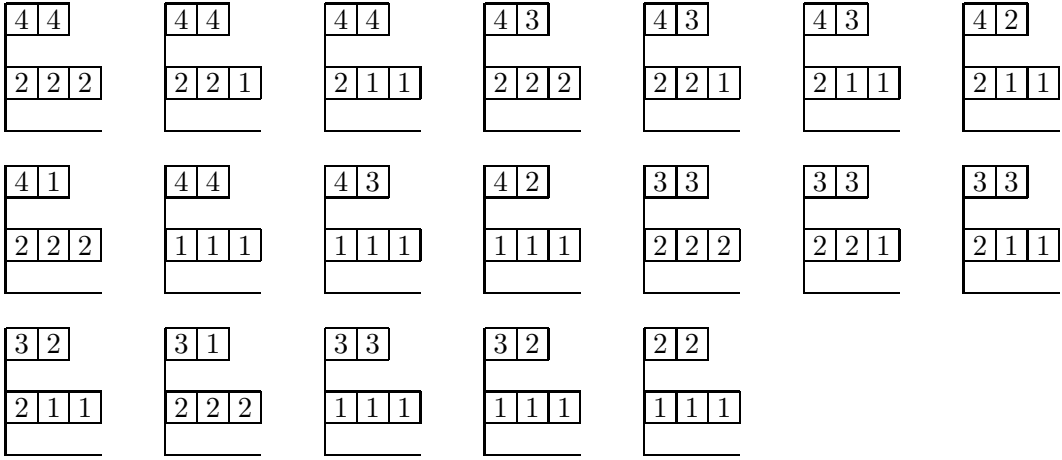

  \begin{center}
    \begin{displaymath}
      \begin{array}{c@{\hskip2\cellsize}c@{\hskip2\cellsize}c@{\hskip2\cellsize}c@{\hskip2\cellsize}c@{\hskip2\cellsize}c@{\hskip2\cellsize}c}
        \vline\tableau{  4 & 4  \\ \\  2 & 2 & 2 \\ \\\hline } &
        \vline\tableau{  4 & 4  \\ \\  2 & 2 & 1 \\ \\\hline } &
        \vline\tableau{  4 & 4  \\ \\  2 & 1 & 1 \\ \\\hline } &
        \vline\tableau{  4 & 3  \\ \\  2 & 2 & 2 \\ \\\hline } &
        \vline\tableau{  4 & 3  \\ \\  2 & 2 & 1 \\ \\\hline } &
        \vline\tableau{  4 & 3  \\ \\  2 & 1 & 1 \\ \\\hline } &
        \vline\tableau{  4 & 2  \\ \\  2 & 1 & 1 \\ \\\hline } \\ \\
        \vline\tableau{  4 & 1  \\ \\  2 & 2 & 2 \\ \\\hline } &
        \vline\tableau{  4 & 4  \\ \\  1 & 1 & 1 \\ \\\hline } &
        \vline\tableau{  4 & 3  \\ \\  1 & 1 & 1 \\ \\\hline } &
        \vline\tableau{  4 & 2  \\ \\  1 & 1 & 1 \\ \\\hline } &
        \vline\tableau{  3 & 3  \\ \\  2 & 2 & 2 \\ \\\hline } &
        \vline\tableau{  3 & 3  \\ \\  2 & 2 & 1 \\ \\\hline } &
        \vline\tableau{  3 & 3  \\ \\  2 & 1 & 1 \\ \\\hline } \\ \\
        \vline\tableau{  3 & 2  \\ \\  2 & 1 & 1 \\ \\\hline } &
        \vline\tableau{  3 & 1  \\ \\  2 & 2 & 2 \\ \\\hline } &
        \vline\tableau{  3 & 3  \\ \\  1 & 1 & 1 \\ \\\hline } &
        \vline\tableau{  3 & 2  \\ \\  1 & 1 & 1 \\ \\\hline } &
        \vline\tableau{  2 & 2  \\ \\  1 & 1 & 1 \\ \\\hline } 
      \end{array}
    \end{displaymath}
    \caption{\label{fig:qKey0302}The 19 elements of $\qKT(0302)$.}
  \end{center}
\end{figure}

\begin{definition}\label{def:qkey}
  Given a weak composition $a$, the \emph{quasi-key polynomial} $\qkey_a$ is defined by
  \begin{equation}\nonumber
    \qkey_a = \sum_{T \in \qKT(a)} x^{\wt(T)}.
  \end{equation}
  \label{def:quasi-key}
\end{definition}

\begin{example}
From Figure~\ref{fig:qKey0302} we compute

\begin{eqnarray}\nonumber
  \qkey_{0302} & = & x^{0302} + x^{1202} + x^{2102} + x^{0311} + x^{1211} + x^{2111} + x^{2201} + x^{1301} + x^{3002} + x^{3011}  \\ \nonumber
 & &  + x^{3101} + x^{0320} + x^{1220} + x^{2120} + x^{2210} + x^{1310} + x^{3020} + x^{3110} + x^{3200}. \nonumber
\end{eqnarray}
\end{example}

\begin{definition}\label{def:qkey1}
  Given a weak composition $a$, the \emph{column quasi-key polynomial} $\qkey^{(1)}_a$ is defined by
  \begin{equation}\nonumber
    \qkey^{(1)}_a = \sum_{T \in \qKT^{(1)}(a)} x^{\wt(T)}.
  \end{equation}
  \label{def:quasi-key1}
\end{definition}

\begin{example}
From the first eight tableaux in Figure~\ref{fig:qKey0302}, we compute
\begin{displaymath}
  \qkey^{(1)}_{0302}  =  x^{0302} + x^{1202} + x^{2102} + x^{0311} + x^{1211} + x^{2111} + x^{2201} + x^{1301}.
\end{displaymath}
\end{example}

\begin{lemma}\label{lem:columnqkey}
\[\qkey_a = \sum_{\substack{b \geq a \\ \flatten(b) = \flatten(a)}} \qkey^{(1)}_b.\]
\end{lemma}
\begin{proof}
Moving each row of boxes to the row indexed by the first-column entry gives a weight-preserving bijection from $\qKT(a)$ to $\bigcup \qKT^{(1)}(b)$ (where $b \geq a$ and $\flatten(b) = \flatten(a)$).
\end{proof}

\begin{example}
\[\qkey_{0302} = \qkey^{(1)}_{0302}+\qkey^{(1)}_{3002}+\qkey^{(1)}_{0320}+\qkey^{(1)}_{3020}+\qkey^{(1)}_{3200}.\]
\end{example}

The quasi-key polynomials expand positively in the fundamental slide basis. The following are direct translations of \cite{AS2} from the language of quasi-Kohnert tableaux to the language of quasi-key tableaux (see Remark~\ref{rmk:quasi-key}).

\begin{definition}[\cite{AS2}]\label{def:quasi-Yamanouchi}
  A quasi-key tableau is \emph{quasi-Yamanouchi} if the leftmost occurrence of each entry $i$ is either
  \begin{enumerate}
  \item in row $i$, or
  \item weakly left of some entry $i+1$.
  \end{enumerate}
  Denote the set of quasi-Yamanouchi quasi-key tableaux of shape $a$ by $\QqKT(a)$.
  \label{def:quasi-Yam}
\end{definition}

\begin{theorem}[\cite{AS2}]
  For a weak composition $a$ of length $n$, we have
    \[\qkey_a = \sum_{T \in \QqKT(a)} \Fund_{\wt(T)}.\]
\end{theorem}

\begin{example}
Only the first, seventh and eighth quasi-Kohnert tableaux in Figure~\ref{fig:qKey0302} are quasi-Yamanouchi. Therefore,
\[\qkey_{0302} = \Fund_{0302} + \Fund_{2201} + \Fund_{1301}.\]
\end{example}

Given a composition $a$, define a \emph{left swap} to be the exchange of two parts $a_i \le a_j$ where $i<j$. 

\begin{definition}[\cite{AS2}]\label{def:lswap}
Given a weak composition $a$, let $\lswap(a)$ be the set of weak compositions that can be obtained via a (possibly empty) sequence of left swaps starting with $a$.
\end{definition}

\begin{example}
\begin{eqnarray}\nonumber
\lswap(0,1,0,3) = \{(0,1,0,3), (0,1,3,0), (1,0,0,3), (1,0,3,0), (1,3,0,0), \\ \nonumber (0,3,0,1), (0,3,1,0), (3,0,0,1), (3,0,1,0), (3,1,0,0)\}.
\end{eqnarray}
\end{example}

Define $\Qlswap(a)$ to be all $b\in \lswap(a)$ such that if $c\in \lswap(a)$ and $\flatten(b) = \flatten(c)$, then $c \ge b$. 

\begin{example}
\[\Qlswap(0,1,0,3) = \{(0,1,0,3), (0,3,0,1)\}.\] 
\end{example}

Then the expansion of a Demazure character into the quasi-key basis is given by

\begin{theorem}[\cite{AS2}]\label{thm:keytoqkey}
\[\key_a = \sum_{b\in \Qlswap(a)} \qkey_b.\]
\end{theorem}


\section{A Littlewood-Richardson rule for quasi-key polynomials}\label{sec:LR}


\subsection{Demazure atom expansion}

We give an explicit positive formula for the Demazure atom expansion of a quasi-key polynomial. Given a weak composition $a$, define a polynomial 
\[\overline\qkey_a = \sum_{\stackrel{b\ge a}{\flatten(b)=\flatten(a)}}\atom_b.\]

We will show the exact same formula as that of Theorem~\ref{thm:keytoqkey} gives the expansion of a Demazure character into $\{\overline\qkey_a\}$. Recall (\ref{eqn:keytoatom}) the expansion of a Demazure character into Demazure atoms
\[\key_a = \sum_{\stackrel{\sort(b)=\sort(a)}{v(b)\le w}} \atom_b.\]
This formula can be straightforwardly re-expressed in terms of $\lswap(a)$ (Definition~\ref{def:lswap}).

\begin{lemma}\label{lem:lswap}
The Demazure atom expansion of a Demazure character is given by
\[\key_a = \sum_{b\in \lswap(a)} \atom_b.\]
\end{lemma}
\begin{proof}
Let $\Shuffle(a)$ denote the set of all distinct rearrangements of the parts of $a$, including $0$ parts. Given $b\in \Shuffle(a)$, we need to show that $b\in \lswap(a)$ if and only if $v(b)\le w=v(a)$. Notice that $v(b)$ is the permutation whose $i$th entry is the position where $b_i$ ends up when $b$ is sorted into decreasing order, where for any pair of identical entries in $b$, the leftmost one is considered greater. For example, if $b=(0,1,0,3)$ then $v(b) = 3241$ (and not $4231$).

Suppose $b\in \lswap(a)$. Inductively, we may assume that $b$ is obtained from $a$ by a single left swap, say swapping entries $a_i<a_j$ with $i<j$, where all entries $a_k$ between $a_i$ and $a_j$ are either strictly smaller than $a_i$ or strictly larger than $a_j$. Then $v(b)$ is obtained by interchanging $w_i$ and $w_j$. Since $a_i<a_j$ we have $w_i>w_j$, and since all entries $a_k$ between $a_i$ and $a_j$ are either strictly smaller than $a_i$ or strictly larger than $a_j$, the same statement holds replacing $a_i, a_j, a_k$ with $w_i, w_j, w_k$. Hence $v(b)\le w$ (in fact, $w$ covers $v(b)$) in Bruhat order.

Conversely, suppose $b\in \Shuffle(a)$ and $v(b)\le w$. Then there is a chain in Bruhat order from $w$ down to $v(b)$. Going down one step in Bruhat order corresponds to interchanging $w_i$ and $w_j$ where $i<j$, $w_i>w_j$, and there is no $i<j<k$ such that $w_i<w_j<w_k$. The weak composition sorted by this new permutation is thus obtained by interchanging $a_i$ and $a_j$, and since $w_i>w_j$, necessarily $a_i\le a_j$. In particular, this is a left swap on $a$. Iterating by moving down a chain in Bruhat order, we find that at each step, the weak composition sorted by the current permutation is obtained from a left swap on the preceding weak composition, hence $b\in \lswap(a)$.
\end{proof}

\begin{lemma}\label{lem:Qlswap}
The $\{\overline{\qkey}\}$ expansion of a Demazure character is given by
\[\key_a = \sum_{b\in \Qlswap(a)} \overline\qkey_b.\]
\end{lemma}
\begin{proof}
Observe that if $b\in \Qlswap(a)$ and $c\ge b$ with $\flatten(c) = \flatten(b)$, then $c\in \lswap(a)$. Moreover by definition, every $c\in \lswap(a)$ is either in $\Qlswap(a)$ or dominates some $b\in \Qlswap(a)$ that satisfies $\flatten(c)=\flatten(b)$. Therefore, we have the following refinement of Lemma~\ref{lem:lswap}:
\[\key_a = \sum_{b\in \Qlswap(a)} \sum_{\stackrel{c\ge b}{\flatten(c)=\flatten(b)}}\atom_c.\]
The statement then follows from the definition of $\overline\qkey_b$.
\end{proof}

\begin{lemma}
The polynomials $\{\overline\qkey_b\}$ form a basis of polynomials.
\end{lemma}
\begin{proof}
By Lemma~\ref{lem:Qlswap}, the polynomials $\{\overline\qkey_b\}$ generate the Demazure characters, hence they span the polynomial ring. They are indexed by the same set (weak compositions) as the Demazure characters, hence they are a basis. 
\end{proof}

\begin{theorem}\label{thm:atomexpansion}
Quasi-key polynomials expand positively in Demazure atoms, via the formula
\[\qkey_a = \sum_{\stackrel{b\ge a}{\flatten(b)=\flatten(a)}}\atom_b.\]
\end{theorem}
\begin{proof}
By Theorem~\ref{thm:keytoqkey} and Lemma~\ref{lem:Qlswap}, the transition matrix between $\{\key_a\}$ and $\{\qkey_b\}$ is identical to the transition matrix between $\{\key_a\}$ and $\{\overline{\qkey}_b\}$. Thus $\overline\qkey_b = \qkey_b$ for all weak compositions $b$.
\end{proof}

The quasi-key basis contains the \emph{quasi-Schur} basis of quasisymmetric polynomials introduced by \cite{HLMvW11a}. Thus we recover the Demazure atom expansion of a quasi-Schur polynomial.

\begin{corollary}[\cite{HLMvW11a}] Let $\alpha$ be a strong composition. The Demazure atom expansion of the quasi-Schur polynomial $QS_\alpha(x_1,\ldots , x_n)$ is given by
\[QS_\alpha(x_1,\ldots , x_n) = \sum_{\flatten(b)=\alpha}\atom_b,\]
where $b$ is a weak composition of length $n$.
\end{corollary}
\begin{proof}
Theorem 4.14 of \cite{AS2} establishes that 
\[QS_{\alpha}(x_1,\ldots , x_n) = \qkey_{0^{n-\ell(\alpha)}\times \alpha},\]
where $\ell(\alpha)$ is the number of entries in $\alpha$. The statement then follows from Theorem~\ref{thm:atomexpansion}.  
\end{proof}

Similarly, we can prove the column quasi-key polynomials (Definition~\ref{def:qkey1}) are precisely the Demazure atoms.

\begin{theorem}\label{thm:colquasikeyisatom}
For any weak composition $a$, we have $\atom_a = \qkey^{(1)}_a$.
\end{theorem}
\begin{proof}
Theorem~\ref{thm:atomexpansion} establishes that
\[\qkey_a = \sum_{\stackrel{b\ge a}{\flatten(b)=\flatten(a)}}\atom_b.\]
On the other hand, by Lemma~\ref{lem:columnqkey},
\[\qkey_a = \sum_{\stackrel{b\ge a}{\flatten(b)=\flatten(a)}}\qkey^{(1)}_b.\]
Therefore, the transition matrix between $\{\qkey_a\}$ and $\{\atom_b\}$ is identical to the transition matrix between $\{\qkey_a\}$ and $\{\qkey^{(1)}_b\}$. Hence $\atom_b = \qkey^{(1)}_b$ for all $b$.
\end{proof}

As a result, column quasi-key polynomials (and thus quasi-key polynomials) may be expressed in terms of semi-skyline fillings as well as quasi-key tableaux, as stated in Theorem C.

\begin{corollary}
\[\qkey_a = \sum_{S\in \SSF(a)} x^{\wt(S)}\]
where entries in the first column of $S$ are at most their row index and decrease from top to bottom.
\end{corollary}

The combinatorial connections between semi-skyline fillings and quasi-key tableaux will be explored in greater depth in Section~\ref{sec:bijection}.

\subsection{Littlewood-Richardson rule}

We now give a positive combinatorial formula for the expansion of the product of a quasi-key polynomial and a Schur polynomial in the quasi-key basis.

We recall the definition of \emph{Littlewood-Richardson skew skyline tableaux} ($\LRS$) from \cite{HLMvW11b}. Let $a$ and $b$ be weak compositions of length $n$ with $a_i\le b_i$ for all $i$. Then an $\LRS$ of shape $b/a$ is a filling of $D(b)$ such that both the basement and the boxes of $D(a)\subseteq D(b)$ are filled with asterisks ``$\ast$'' and the remaining boxes are filled with positive integers such that the filling weakly decreases along rows, does not repeat an entry in any column, and all triples are inversion triples.  (In determining the status of triples involving $\ast$ symbols, following \cite{HLMvW11b} we declare that $\ast=\infty$, $\ast$ symbols in the same row are equal, and $\ast$ symbols in the same column increase from top to bottom.) For example, see Figure~\ref{fig:LRS0103}.

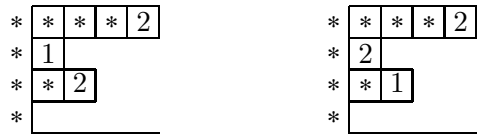
\begin{figure}[ht]
\begin{center}
\begin{picture}(192,48)

\put(3.5,3.5){$\ast$}
\put(3.5,15.5){$\ast$}
\put(3.5,27.5){$\ast$}
\put(3.5,39.5){$\ast$}

\put(15.5,15.5){$\ast$}
\put(27.5,15.5){$2$}

\put(15.5,26.5){$1$}

\put(15.5,39.5){$\ast$}
\put(27.5,39.5){$\ast$}
\put(39.5,39.5){$\ast$}
\put(51.5,38.5){$2$}

\put(12,0){\line(1,0){48}}
\put(12,0){\line(0,1){48}}

\put(12,12){\line(1,0){24}}
\put(12,24){\line(1,0){24}}
\put(12,36){\line(1,0){48}}
\put(12,48){\line(1,0){48}}

\put(24,12){\line(0,1){36}}
\put(36,12){\line(0,1){12}}
\put(36,36){\line(0,1){12}}
\put(48,36){\line(0,1){12}}
\put(60,36){\line(0,1){12}}




\put(123.5,3.5){$\ast$}
\put(123.5,15.5){$\ast$}
\put(123.5,27.5){$\ast$}
\put(123.5,39.5){$\ast$}

\put(135.5,15.5){$\ast$}
\put(147.5,15.5){$1$}

\put(135.5,26.5){$2$}

\put(135.5,39.5){$\ast$}
\put(147.5,39.5){$\ast$}
\put(159.5,39.5){$\ast$}
\put(171.5,38.5){$2$}

\put(132,0){\line(1,0){48}}
\put(132,0){\line(0,1){48}}

\put(132,12){\line(1,0){24}}
\put(132,24){\line(1,0){24}}
\put(132,36){\line(1,0){48}}
\put(132,48){\line(1,0){48}}

\put(144,12){\line(0,1){36}}
\put(156,12){\line(0,1){12}}
\put(156,36){\line(0,1){12}}
\put(168,36){\line(0,1){12}}
\put(180,36){\line(0,1){12}}

\end{picture}
\caption{\label{fig:LRS0103}Two $\LRS$ of shape $(0,2,1,4)/(0,1,0,3)$.}
\end{center}
\end{figure}


Given an $\LRS$ $L$, form the \emph{column word} of $L$ by reading entries of $L$ from bottom to top in each column, starting with the rightmost column and working leftwards, and ignoring asterisks. A column word whose largest letter is $k$ is said to be \emph{contre-lattice} if for all $i$, the subword consisting of the first $i$ letters has at least as many $k$'s as $k-1$'s, at least as many $k-1$'s as $k-2$'s, etc. For example, $4432314213$ is contre-lattice. Let $\lambda = \lambda_1\ge \ldots \ge \lambda_\ell>0$ be a partition of length $\ell$ and let $\lambda^\ast$ be the multiset of numbers $1^{\lambda_\ell} 2^{\lambda_{\ell-1}} \ldots \ell^{\lambda_1}$.  

With this, \cite{HLMvW11b} give the following Littlewood-Richardson rule for the Demazure atom expansion of the product of a Demazure atom and a Schur polynomial. 

\begin{theorem}[\cite{HLMvW11b}]\label{thm:atomxschur}
Let $a$ be a weak composition of length $n$ and $\lambda$ a partition. Then
\[\atom_a \cdot s_\lambda(x_1,\ldots , x_n) = \sum_b c_{a,\lambda}^b\atom_b,\]
where $b$ is a weak composition of length $n$ and $c_{a,\lambda}^b$ is the number of $\LRS$ of shape $b/a$, content $\lambda^\ast$.
\end{theorem}

We now define operators $\swap_{i,i+1}$ and $\swap_{i, i-1}$ on $\LRS$. Let $L$ be an $\LRS$ with an occupied row $i$. If row $i+1$ is empty, let $\swap_{i, i+1}(L)$ be the diagram obtained by moving all entries of row $i$ up to row $i+1$; similarly if row $i-1$ is empty, let $\swap_{i, i-1}(L)$ be the diagram obtained by moving all entries of row $i$ down to row $i-1$. For example, see Figure~\ref{fig:swap}.

\begin{figure}[ht]
\begin{center}
\begin{picture}(288,48)

\put(3.5,3.5){$\ast$}
\put(3.5,15.5){$\ast$}
\put(3.5,27.5){$\ast$}
\put(3.5,39.5){$\ast$}

\put(15.5,14.5){$2$}
\put(27.5,14.5){$1$}

\put(15.5,3.5){$\ast$}

\put(15.5,39.5){$\ast$}
\put(27.5,39.5){$\ast$}
\put(39.5,39.5){$\ast$}
\put(51.5,38.5){$2$}

\put(12,0){\line(1,0){48}}
\put(12,0){\line(0,1){48}}

\put(12,12){\line(1,0){24}}
\put(12,24){\line(1,0){24}}
\put(12,36){\line(1,0){48}}
\put(12,48){\line(1,0){48}}

\put(24,0){\line(0,1){24}}
\put(24,36){\line(0,1){12}}

\put(36,12){\line(0,1){12}}
\put(36,36){\line(0,1){12}}
\put(48,36){\line(0,1){12}}
\put(60,36){\line(0,1){12}}




\put(111.5,3.5){$\ast$}
\put(111.5,15.5){$\ast$}
\put(111.5,27.5){$\ast$}
\put(111.5,39.5){$\ast$}

\put(123.5,26.5){$2$}
\put(135.5,26.5){$1$}

\put(123.5,3.5){$\ast$}

\put(123.5,39.5){$\ast$}
\put(135.5,39.5){$\ast$}
\put(147.5,39.5){$\ast$}
\put(159.5,38.5){$2$}

\put(120,0){\line(1,0){48}}
\put(120,0){\line(0,1){48}}

\put(120,12){\line(1,0){12}}
\put(120,24){\line(1,0){24}}
\put(120,36){\line(1,0){48}}
\put(120,48){\line(1,0){48}}

\put(132,0){\line(0,1){12}}
\put(132,24){\line(0,1){24}}
\put(144,24){\line(0,1){24}}
\put(156,36){\line(0,1){12}}
\put(168,36){\line(0,1){12}}



\put(219.5,3.5){$\ast$}
\put(219.5,15.5){$\ast$}
\put(219.5,27.5){$\ast$}
\put(219.5,39.5){$\ast$}

\put(231.5,26.5){$2$}
\put(243.5,26.5){$1$}

\put(231.5,15.5){$\ast$}

\put(231.5,39.5){$\ast$}
\put(243.5,39.5){$\ast$}
\put(255.5,39.5){$\ast$}
\put(267.5,38.5){$2$}

\put(228,0){\line(1,0){48}}
\put(228,0){\line(0,1){48}}

\put(228,12){\line(1,0){12}}
\put(228,24){\line(1,0){24}}
\put(228,36){\line(1,0){48}}
\put(228,48){\line(1,0){48}}

\put(240,12){\line(0,1){36}}
\put(252,24){\line(0,1){24}}
\put(264,36){\line(0,1){12}}
\put(276,36){\line(0,1){12}}

\end{picture}
  \caption{\label{fig:swap}An $\LRS$ $L$ (left), $L'=\swap_{2,3}(L)$ (middle), $L''=\swap_{1,2}(L')$ (right).}
\end{center}
\end{figure}
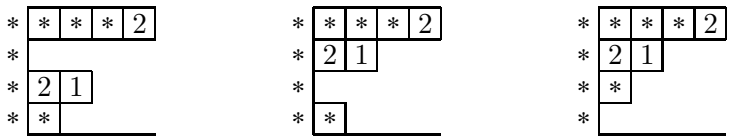

The following will play an important role in the proof of our Littlewood-Richardson rule.

\begin{lemma}\label{lem:rowtranslate}
If $L$ is an $\LRS$, then $\swap_{i, i+1}(L)$ and $\swap_{i, i-1}(L)$ also satisfy the triple conditions, and have a contre-lattice column word if and only if the column word of $L$ is contre-lattice.
\end{lemma}
\begin{proof}
Translating occupied rows (while retaining their relative order) does not change the column word or the status of any triple of entries.
\end{proof}

Given weak compositions $a$ and $b$ of length $n$, define $\LRS(a,b)$ to be the set of $\LRS$ of shape $b/c$ where $c$ is any weak composition of length $n$ satisfying $c\ge a$ and $\flatten(c) = \flatten(a)$. Say an element $L\in \LRS(a,b)$ is \emph{highest-weight} if for every row $i$ of $L$, $\swap_{i, i+1}(L)$ is not an element of $\bigcup_d \LRS(a,d)$, where $d$ ranges over weak compositions of length $n$. Denote the set of highest-weight elements of $\LRS(a,b)$ by $\HLRS(a,b)$.

\begin{example} Suppose $n=4$ and $a=(0,1,0,3)$. Then both the two $\LRS$ in Figure~\ref{fig:LRS0103} and the rightmost $\LRS$ in Figure~\ref{fig:swap} are highest-weight: we would not be able to apply e.g. $\swap_{4,5}$, since the result would no longer be inside $\bigcup_d \LRS(a,d)$. The leftmost two $\LRS$ in Figure~\ref{fig:swap} are not highest-weight.
\end{example}

We may now state our Littlewood-Richardson rule for quasi-key polynomials.

\begin{theorem}
Let $a$ be a weak composition of length $n$ and $\lambda$ a partition. Then
\[\qkey_a\cdot s_\lambda(x_1,\ldots , x_n) = \sum_b C_{a,\lambda}^b\qkey_b,\]
where $b$ is a weak composition of length $n$ and $C_{a,\lambda}^b$ is the number of $\HLRS(a,b)$ with content $\lambda^\ast$.
\end{theorem}

\begin{proof}
By Theorem~\ref{thm:atomexpansion}, the left-hand side is equal to 
\[\sum_{\substack{c\ge a \\ \flatten(c)=\flatten(a)}}\atom_c\cdot s_\lambda(x_1,\ldots x_n).\]

By Theorem~\ref{thm:atomxschur}, the Demazure atom expansion of this expression is indexed by the elements of $\bigcup_d \LRS(a,d)$ (where $d$ ranges over weak compositions of length $n$) with content $\lambda^\ast$ and contre-lattice column word. We need to collect these Demazure atoms into quasi-key polynomials.

Let $L\in \bigcup_d \LRS(a,d)$ be an $\LRS$ indexing a Demazure atom appearing in the expansion of the left hand side. Consider the closure of the action of the operators $\swap_{i,i+1}$ and $\swap_{i,i-1}$ on $L$, with the restriction that these operators may be applied only if the image remains in $\bigcup_d \LRS(a,d)$. By Lemma~\ref{lem:rowtranslate}, everything thus obtained is an element of $\bigcup_d \LRS(a,d)$ with content $\lambda^\ast$ and contre-lattice column word, hence indexes an atom in the expansion of the the left-hand side. Since the swap operators do not change the rows of an $\LRS$, but only move them up and down, this action of the swap operators partitions $\bigcup_d \LRS(a,d)$ into equivalence classes.

Consider such an equivalence class. The $\LRS$ whose underlying composition is smallest in dominance order (say $b$) in this class is by definition an element of $\HLRS(a,b)$, and each element of $\HLRS(a,b)$ belongs to exactly one equivalence class. Moreover, by the definition of the swap operators, the underlying weak compositions of the elements of an equivalence class are exactly those $b'$ satisfying $b'\ge b$ and $\flatten(b')=\flatten(b)$. Hence the sum of all atoms in the class is $\qkey_b$.
\end{proof}

%
\section{The fundamental particle basis of polynomials}\label{sec:particle}
%
Completing the picture of $\mathcal{P}$, and further connecting Demazure atoms and fundamental slide polynomials, we introduce the \emph{fundamental particle} basis $\{\Part_a\}$ of the polynomial ring. 

By Theorem~\ref{thm:atomexpansion} we have 
\[\qkey_a = \sum_{\substack{b\ge a \\ \flatten(b) = \flatten(a)}}\atom_b\]
and clearly monomial slide polynomials expand into ordinary monomials via the same rule:
\[\Mono_a = \sum_{\substack{b\ge a \\ \flatten(b) = \flatten(a)}}x^b.\]
The fundamental slide polynomials will expand via the same rule into fundamental particles.

\begin{definition}\label{def:slidemoves}
Let $a$ be a weak composition. A \emph{slide} of $a$ is a weak composition that can be obtained from $a$ via a sequence of local moves of the form
\[\ldots 0k \ldots \rightarrow \ldots ij \ldots\]
where $i+j = k$ and $i,j\ge 0$. A \emph{fixed slide} of $a$ is a weak composition that can be obtained from $a$ via a sequence of the above local moves subject to an additional condition:  that if $k$ occupies a position that is nonzero in $a$, then $j>0$. Let $\Slide(a)$ denote the set of all slides of $a$ and $\FS(a)$ the set of all fixed slides of $a$.
\end{definition}

\begin{example}
Let $a=(0,3,0,2)$. Then
\begin{align}\nonumber
\Slide(a)  = & \{(0,3,0,2), (1,2,0,2), (2,1,0,2), (0,3,1,1), (1,2,1,1), (2,1,1,1), (3,0,0,2), (3,0,1,1), \\ \nonumber & (3,1,0,1), (0,3,2,0), (1,2,2,0), (2,1,2,0), (3,0,2,0), (3,1,1,0), (3,2,0,0)\} \\ \nonumber
\FS(a)  = & \{(0,3,0,2), (1,2,0,2), (2,1,0,2), (0,3,1,1), (1,2,1,1), (2,1,1,1)\}.
\end{align}
\end{example}

The following is proved in \cite[Proposition 2.16]{PS17}.

\begin{lemma}
Let $a$ be a weak composition. The fundamental slide polynomial $\Fund_a$ is the generating function of the slides of $a$, i.e.,
\[\Fund_a = \sum_{b\in \Slide(a)}x^b.\]
\end{lemma}

\begin{definition}
Given a weak composition $a$, the fundamental particle $\Part_a$ is the generating function of the fixed slides of $a$, i.e.,
\[\Part_a = \sum_{b\in\FS(a)}x^b.\]
\end{definition}

\begin{example}
$\Part_{0302} = x^{0302} + x^{1202} + x^{2102} + x^{0311} + x^{1211} + x^{2111}.$
\end{example}

\begin{proposition}
$\{\Part_a\}$ is a basis for the polynomial ring.
\end{proposition}
\begin{proof}
The leading monomial of $\Part_a$ with respect to dominance order is $x^a$, so we have a triangularity with the basis of (ordinary) monomials.
\end{proof}

\subsection{Positivity}

Fundamental slide polynomials expand positively in the fundamental particle basis:

\begin{proposition}\label{prop:fundtoparticle}
Let $a$ be a weak composition. Then
\[\Fund_a = \sum_{\stackrel{b\ge a}{\flatten(b)=\flatten(a)}}\Part_b.\]
\end{proposition}
\begin{proof}
Let $f\in \Slide(a)$. We first want to show that $f\in \FS(b)$ for some $b\ge a$ with $\flatten(b)=\flatten(a)$. Suppose $a$ has nonzero entries in positions $n_1 < n_2 < \cdots < n_k$. Then by definition, there are positions $i_j<n_j$ such that $a_{n_1} = f_1 + \cdots + f_{i_1}$, $a_{n_2} = f_{i_1+1} + \cdots + f_{i_2}$, etc. Perform local moves of Definition~\ref{def:slidemoves} on $a$ that move the nonzero entries of $a$ from position $n_j$ to position $i_j$ for each $1 \le j \le k$, in order. Let $b$ be the weak composition thus obtained; clearly $b\ge a$ and $\flatten(b)=\flatten(a)$. Then we may perform further local moves on $b$ to obtain $f$, and all these moves satisfy the condition for fixed slides. Hence $f\in \FS(b)$, establishing that every term on the left hand side also appears on the right hand side.

To show that every term on the right hand side also appears on the left hand side, first observe that if $f\in \FS(b)$ where $b\ge a$ and $\flatten(b)=\flatten(a)$, then clearly $f\in \Slide(a)$. It remains to show that if $f\in \Slide(a)$ then there is no more than one $b$ such that $\flatten(b) = \flatten(a)$ and $f\in \FS(b)$. Suppose for a contradiction that $f\in \FS(b)$ and $f\in \FS(c)$, for some $b\neq c$ where $\flatten(c)=\flatten(b)$. Suppose $b$ has nonzero entries in positions $n_1 < n_2 < \cdots < n_\ell$ and $c$ has nonzero entries in positions $m_1 < m_2 < \cdots < m_\ell$.  
Let $j$ be the smallest index such that $n_j \neq m_j$; without loss of generality assume $n_j<m_j$. Then we have
\[f_{n_{j-1}+1} + \cdots + f_{n_j} = b_{n_j} = c_{m_j} = f_{m_{j-1}+1} + \cdots + f_{m_j},\]
where the middle equality is due to $\flatten(b) = \flatten(c)$. Since by assumption $n_{j-1} = m_{j-1}$, by subtracting the left hand side from the right hand side we obtain
\[0 = f_{n_j+1} + \cdots + f_{m_j}.\]
This implies in particular that $f_{m_j} = 0$. Since $c_{m_j} \neq 0$, this contradicts $f\in \FS(c)$.
\end{proof}

\begin{example}
$\Fund_{0302} = \Part_{0302} + \Part_{3002} + \Part_{0320} + \Part_{3020} + \Part_{3200}.$
\end{example}

We now prove Theorem A for bases incomparable in $\mathcal{P}$.

\begin{proposition}\label{prop:nonpositive}
For the three pairs of bases $(\{A_a\}, \{\Mono_a\}), (\{A_a\}, \{\Fund_a\}), (\{\Part_a\}, \{\Mono_a\})$, neither basis in the pair expands positively in the other.
\end{proposition}
\begin{proof}
Observe that
\begin{align*}
\atom_{01} & =  \Mono_{01} - \Mono_{10} & \Mono_{02} & =  \atom_{02} + \atom_{20} - \atom_{11} \\
\atom_{01} & =  \Fund_{01} - \Fund_{10} &  \Fund_{13} & =  \atom_{13} - \atom_{22} \\
\Part_{01}  & =   \Mono_{01} - \Mono_{10}  & \Mono_{02} & =  \Part_{02} + \Part_{20} - \Part_{11}.
\end{align*}
\end{proof}

Completing the proof of Theorem C, we now show fundamental particles can be described in terms of semi-skyline fillings. 

\begin{definition}
Given a weak composition $a$, let $\LSSF(a)$ be the set of all semi-skyline fillings of $D(a)$ such that the first column entries are equal to their row index, and if a box $B$ is in a higher row than a box $B'$, then the entry of $B$ is strictly larger than the entry of $B'$. 
\end{definition}

Notice that these conditions force all triples to be inversion triples, as in Remark~\ref{rmk:fundtriple}.

\begin{proposition}
The fundamental particle $\Part_a$ has the following expansion:
\[\Part_a = \sum_{T\in \LSSF(a)}x^{\wt(T)}.\]
\end{proposition}
\begin{proof}
By definition, the weight $\wt(T)$ of $T\in \LSSF(a)$ is a fixed slide of $a$. Conversely, for every $b\in \FS(a)$, one can associate a unique $T\in \LSSF(a)$ as follows. Suppose $a$ has nonzero entries in positions $n_1, \ldots , n_k$, and set $n_0=0$. Then for each $1\le j \le k$, we have $b_{n_{j-1}+1} + \cdots + b_{n_j} = a_{n_j}$. For each $1\le j \le k$, fill the rightmost $b_{n_{j-1}+1}$  boxes of row $n_j$ of $D(a)$ with entry $n_{j-1}+1$, then, proceeding leftward, the next $b_{n_{j-1}+2}$ boxes with entry $n_{j-1}+2$, etc, finishing with the leftmost $b_{n_j}$ boxes being filled with entry $n_j$. 
\end{proof}

The following is immediate from the definitions:

\begin{proposition}
The semi-skyline fillings of $D(a)$ that generate the fundamental particle $\Part_a$ are the intersection of those that generate the Demazure atom $\atom_a$ and those that generate the fundamental slide $\Fund_a$, i.e.,
\[\LSSF(a) = \ASSF(a)\cap\FSSF(a).\]
\end{proposition}

Demazure atoms also expand positively into fundamental particles. To give this expansion, we introduce a distinguished subset of $\ASSF(a)$. 

\begin{definition} A semi-skyline filling $T\in \ASSF(a)$ is \emph{particle-highest} if for every $i$ that appears as an entry in $T$, either 
\begin{itemize}
\item the leftmost $i$ is in the first column, or 
\item there is an $i^\uparrow$ in some column weakly to the right of the leftmost $i$, where $i^\uparrow$ is the smallest label greater than $i$ appearing in $T$. 
\end{itemize}
Let $\HSSF(a)$ denote the set of particle-highest elements of $\ASSF(a)$.
\end{definition}

\begin{example}
Only the first and third $\ASSF$s in Figure~\ref{fig:atom0103} are in $\HSSF(0,1,0,3)$.
\end{example}

In order to give a formula for the fundamental particle expansion of a Demazure atom, we now define a \emph{destandardization} operation $\destand$ on $\ASSF(a)$. For each entry $i$ that appears in $T\in \ASSF(a)$, if the leftmost $i$ is not in the first column, and it has no $i^\uparrow$ weakly to its right, then replace every $i$ in $T$ with an $i+1$. Repeat until no further changes can be made: the result is $\destand(T)$. This process terminates since it increases entries, but the entry of each box is bounded above by its row index.

\begin{lemma}\label{lem:dst}
If $T\in \ASSF(a)$ then $\destand(T)\in \HSSF(a)$, and $\destand(T)=T$ if and only if $T\in \HSSF(a)$. 
\end{lemma}
\begin{proof}
We must show that applying $\destand$ preserves the $\ASSF$ conditions. Suppose we apply one step of destandardization to entry $i$. 

We apply $\destand$ to $i$ only if there is no $i^\uparrow$ in a column weakly right of the leftmost $i$. Hence no $i$ has a $i^\uparrow$ in its column, so in particular there is no column containing both $i$ and $i+1$. Thus we never introduce a repeat of any entry in any column.

If row $i$ contains an entry $i$, then since rows weakly decrease and row entries do not exceed their row index, this row has an $i$ in the first column, so the leftmost $i$ is in the first column. Therefore we may assume all entries $i$ in $T$ are in rows with indices strictly greater than $i$. Moreover, the label immediately left of the leftmost $i$ in each row is strictly greater than $i$, hence replacing all $i$'s with $i+1$'s preserves weakly decreasing rows, and the entry of each box remains weakly smaller than its row index. 

To see no type A coinversion triples are introduced, suppose we have a type $A$ inversion triple with top entry strictly smallest. This could become a coinversion triple under this process only if the top entry is $i$ and the entry below in the same column is $i+1=i^\uparrow$, but then we have an $i$ and $i+1$ in the same column and $\destand$ does not apply to $i$. Now suppose we have a type $A$ inversion triple with top entry strictly largest. This could become become a coinversion triple under this process only if the top entry is $i+1=i^\uparrow$ and the bottom-left entry is $i$. But then there is an $i+1$ to the right of an $i$, so again $\destand$ does not apply to $i$.

To see no type B coinversion triples are introduced, suppose we have a type $B$ inversion triple with bottom entry strictly smallest. This could become a coinversion triple under this process only if the bottom entry is $i$ and the top-right entry is $i+1=i^\uparrow$, but then there is an $i+1$ to the right of an $i$, and $\destand$ does not apply to $i$. Now suppose we have a type $B$ inversion triple with bottom entry strictly largest. This could become a coinversion triple under this process only if the bottom entry is $i+1=i^\uparrow$ and the top entry in the same column is $i$, but then we have an $i$ and $i+1$ in the same column and $\destand$ does not apply to $i$.

Thus $\destand(T)$ is an $\ASSF$. The fact that $\destand(T)\in \HSSF(a)$ and that $\destand(T)=T$ if and only if $T\in \HSSF(a)$ now follows immediately from the definitions.
\end{proof}

\begin{example}
In Figure~\ref{fig:atomdst}, the leftmost $\ASSF$ is not particle-highest and destandardizes to the particle-highest $\ASSF$ in the middle, while the rightmost $\ASSF$ is particle-highest (and destandardizes to itself).
\end{example}

\begin{figure}[ht]
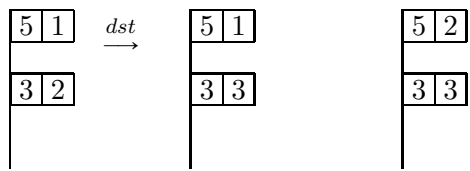

  \begin{center}
    \begin{displaymath}
      \begin{array}{c@{\hskip\cellsize}c@{\hskip\cellsize}c}
        \vline\tableau{ 5 & 1 \\ \\ 3 & 2 \\ \\ \\ } \,\,\,\,\,\, \substack{dst \\ \longrightarrow} \,\,\,\, &
        \vline\tableau{ 5 & 1 \\ \\ 3 & 3  \\ \\ \\ } \qquad \qquad &
        \vline\tableau{ 5 & 2 \\ \\ 3 & 3 \\ \\ \\ } 
      \end{array}
    \end{displaymath}
    \caption{\label{fig:atomdst} Three elements of $\ASSF(00202)$.}
  \end{center}
\end{figure}

We may now give a positive combinatorial formula for the expansion of a Demazure atom into fundamental particles, completing the proof of Theorem A.

\begin{theorem}\label{thm:atomtoparticle} Let $a$ be a weak composition of length $n$. Then
\[\atom_a = \sum_{S\in \HSSF(a)}\Part_{\wt(S)}.\]
\end{theorem}
\begin{proof}
Let $T\in \ASSF(a)$ and suppose $\destand(T)=S$. We first show the monomial $x^{\wt(T)}$ belongs to $\Part_{\wt(S)}$.  
Note that if $\destand(T)=S$, then $S$ is obtained from $T$ by a series of moves in which \emph{all} $i$'s are changed to $i+1$'s, and no entry in the first column is ever altered. The corresponding moves taking $\wt(T)$ to $\wt(S)$ are therefore inverses of the moves of Definition~\ref{def:slidemoves} defining fixed slides, hence $\wt(T)$ is a fixed slide of $\wt(S)$.

Conversely, suppose $S\in \HSSF(a)$. We claim that for every fixed slide $b$ of $\wt(S)$, there is a unique $T\in \ASSF(a)$ with $\wt(T)=b$ and $\destand(T)=S$.

This claim implies
\[\Part_{\wt(S)}=\sum_{T\in \destand^{-1}(S)}x^{\wt(T)},\]
and the theorem follows from this and Lemma~\ref{lem:dst}.

To construct $T$ from $S$ and $b$, suppose for some $i<j$ we have $\wt(S)_i>0$ and $\wt(S)_j>0$ but $\wt(S)_{i+1}=\ldots = \wt(S)_{j-1}=0$. Then for any such interval of zeros in $\wt(S)$, $T$ is constructed by changing the rightmost $b_{i+1}$ $j$'s to $i+1$, the next $b_{i+2}$ $j$'s to $i+2$, etc, and the leftmost $b_j$ $j$'s remain $j$'s. 
Note that since $b$ is a fixed slide of $\wt(S)$, we have $b_j>0$. 
By construction, $T$ has weight $b$ and $\destand(T)=S$. To see $T\in \ASSF(a)$, note the leftmost $j$ of $S$ remains a $j$ when constructing $T$, and clearly the triple conditions, weakly decreasing rows and no repeated entries in a column are preserved.  Existence is proved, and uniqueness follows from lack of choice at every step.
\end{proof}

\begin{example}
The Demazure atom $\atom_{0103}$ expands in fundamental particles as follows.
\[\atom_{0103} = \Part_{0103} + \Part_{0202}.\]
\end{example}

\subsection{Littlewood-Richardson rule}

Like the Demazure atom basis, the fundamental particle basis does not have positive structure constants. However, the product of a fundamental particle and a Schur polynomial does expand positively in the basis of fundamental particles. In this section, we give a positive Littlewood-Richardson rule for this expansion, completing the proof of Theorem~B.

Let $\revSSYT_n$ denote the set of reverse tableaux of shape $\lambda$ (i.e., fillings of the Young diagram for $\lambda$ that weakly decrease along rows and strictly decrease down columns), whose largest entry is at most $n$. Let $a$ be a weak composition of length $n$. By definition, the monomials appearing in the product $\Part_a\cdot s_\lambda(x_1,\ldots , x_n)$ arise from the pairs $(S,T)$ where $S \in \LSSF(a)$, $T\in \revSSYT_n(\lambda)$. Denote the set of such pairs by $\Pairs(a, \lambda)$, and for $(S,T)\in \Pairs(a,\lambda)$, let $\wt(S,T) = \wt(S)+\wt(T)$.
 We extend the definition of destandardization to $\Pairs(a, \lambda)$ by considering every label in $T$ to be strictly right of every label in $S$ when applying $\destand$ to $(S,T)\in \Pairs(a, \lambda)$. Let $\HPairs(a,\lambda)$ be the set $\{\destand(S,T) : (S,T) \in \Pairs(a,\lambda)\}$. For an example, see Figure~\ref{fig:pairs}.

\begin{figure}[ht]
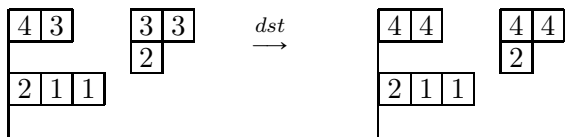

  \begin{center}
    \begin{displaymath}
      \begin{array}{cc@{\hskip\cellsize}cc}
        \vline\tableau{ 4 & 3 \\ \\ 2 & 1 & 1 \\ \\} &
        \tableau{ 3 & 3  \\  2 \\  \\ \\ }  \qquad \substack{dst \\ \longrightarrow} \qquad &
        \vline\tableau{ 4 & 4 \\ \\ 2 & 1 & 1 \\ \\} &
         \tableau{ 4 & 4  \\  2 \\  \\  \\}  
      \end{array}
    \end{displaymath}
    \caption{\label{fig:pairs} Destandardization of an element of $\Pairs((0,3,0,2),(2,1))$, where $n=4$.}
  \end{center}
\end{figure}

\begin{lemma}\label{lem:destandpairs}
If $(S,T)\in \Pairs(a,\lambda)$, then $\destand(S,T)\in \Pairs(a,\lambda)$, and $\destand(S,T)=(S,T)$ if and only if $(S,T)\in \HPairs(a,\lambda)$.
\end{lemma} 
\begin{proof}
Consider entries $i$ in $(S,T)$. Suppose the leftmost $i$ is strictly right of the rightmost $i^\uparrow$. The filling $S'$ obtained by changing all entries $i$ in $S$ to $i+1$ is an element of $\LSSF(a)$, since we do not alter any entry of the first column and since we only change $i$'s to $i+1$'s if all $i$'s are strictly right of all $j$'s. Similarly, the filling $T'$ obtained from $T$ is in $T\in \revSSYT_n(\lambda)$ since by assumption no $i$ is in the same column as any $i^\uparrow$ (and thus any $i+1$), and we change the entire horizontal strip of $i$'s in $T$ to $i+1$'s.

By definition, $\destand(S,T)=(S,T)$ if and only if $(S,T)\in \HPairs(a,\lambda)$.
\end{proof}

\begin{theorem}
Let $a$ be a weak composition of length $n$ and $\lambda$ a partition. Then
\[\Part_a\cdot s_\lambda(x_1,\ldots , x_n) = \sum_{(S,T)\in \HPairs(a, \lambda)}\Part_{\wt(S,T)}.\]
\end{theorem}
\begin{proof}
Let $(\overline{S},\overline{T})\in \Pairs(a,\lambda)$ and suppose $\destand(\overline{S},\overline{T}) = (S,T)$. We begin by showing that the monomial $x^{\wt(\overline{S},\overline{T})}$ is a term of $\Part_{\wt(S,T)}$. Note that if $\destand(\overline{S},\overline{T}) = (S,T)$, then $(S,T)$ is obtained from $(\overline{S},\overline{T})$ by a sequence of moves in which \emph{all} $i$'s are changed to $i+1$'s, and no entry in the first column of $(S,T)$ (i.e., the first column of $S$) is ever altered. The corresponding moves taking $\wt(\overline{S},\overline{T})$ to $\wt(S,T)$ are therefore inverses of the moves of Definition~\ref{def:slidemoves} defining fixed slides, hence $\wt(\overline{S},\overline{T})$ is a fixed slide of $\wt(S,T)$.

Conversely, suppose $(S,T)\in \HPairs(a,\lambda)$. We claim that for every fixed slide $b$ of $\wt(S,T)$, there is a unique $(\overline{S},\overline{T})\in \Pairs(a,\lambda)$ with $\wt(\overline{S},\overline{T})=b$ and $\destand(\overline{S},\overline{T}) = (S,T)$.

The claim implies 
\[\Part_{\wt(S,T)} = \sum_{(\overline{S},\overline{T})\in \destand^{-1}(S,T)} x^{\wt(\overline{S},\overline{T})},\]
and the theorem follows from this and Lemma~\ref{lem:destandpairs}.

To construct $(\overline{S},\overline{T})$ from $(S,T)$ and $b$, suppose that for some $i<j$ we have $\wt(S,T)_i>0$, $\wt(S,T)_j>0$ but $\wt(S,T)_{i+1} = \wt(S,T)_{i+2} = \ldots =\wt(S,T)_{j-1} =0$. Then for any such interval of zeros in $\wt(S,T)$, $(\overline{S},\overline{T})$ is constructed by changing the rightmost $b_{i+1}$ $j$'s to $i+1$, then next $b_{i+2}$ $j$'s to $i+2$, etc, and the leftmost $b_j$ $j$'s remain as $j$'s. 
Note that since $b$ is a fixed slide of $\wt(S,T)$, we have $b_j>0$. 
By construction, $(\overline{S},\overline{T})$ has weight $b$ and $\destand(\overline{S},\overline{T})=(S,T)$. To see that $(\overline{S},\overline{T})\in \Pairs(a,\lambda)$, note the leftmost $j$ of $(S,T)$ remains a $j$ when constructing $(\overline{S},\overline{T})$, and it is clear that the remaining $\LSSF$ and $\revSSYT$ conditions are preserved. Existence is proved, and uniqueness follows from the lack of choice at every step.
\end{proof}

%
\section{Bijections between quasi-key tableaux and semi-skyline fillings}\label{sec:bijection}
%
In this section we examine the relationship between semi-skyline fillings and quasi-key tableaux. Let $\ASSF$ denote the set of all $\ASSF(a)$ as $a$ ranges over weak compositions, and define $\qKT^{(1)}$ similarly. Let $\revSSYT$ denote the set of all reverse tableaux of all shapes, and call the set of entries in the $c$'th column of a tableau filling the $c$'th \emph{column set}. We begin by constructing a column-set (and thus weight) preserving bijection between $\ASSF$ and $\qKT^{(1)}$.  
We then use the combinatorics of the fundamental particle basis to restrict this to a bijection between $\HSSF=\{\destand(S): S\in \ASSF\}$ and a similarly-defined subset $\HqKT^{(1)}$ of $\qKT^{(1)}$. By using a different destandardization map, we restrict this bijection yet further to a bijection between a subset $\QSSF$ of $\HSSF$ and $\QqKT^{(1)}$, the set of quasi-Yamanouchi quasi-key tableaux whose first column entries are equal to their row index. Finally, we use a result of \cite{MPS18} to show these bijections preserve the underlying skyline diagram.

\subsection{Row-filling algorithms and bijections to reverse tableaux}

First, we recall the bijection between semi-skyline augmented fillings and reverse tableaux given in \cite{Mason08}. Given $V\in \revSSYT$, create the first column of $S\in \ASSF$ from the first column set of $V$ by placing each entry $i$ in row $i$. Now supposing we have constructed column $c$ of $S$, create column $c+1$ of $S$ from the $(c+1)$'th column set of $V$ by placing the largest entry as low as possible in column $c+1$ of $S$ such that the decreasing row condition remains satisfied, then repeat with the next-largest entry, etc. Call this the \emph{column-filling algorithm} for building an $\ASSF$ from a $\revSSYT$. Note the column-filling algorithm preserves column sets: for all $c$, the $c$'th column set of $V$ is equal to the $c$'th column set of $S$.

\begin{lemma}[\cite{Mason08}]
The column-filling algorithm is a bijection from $\revSSYT$ to $\ASSF$.
\end{lemma}

We introduce an alternative (but equivalent) method that we call the \emph{left row-filling algorithm}. Given $V\in \revSSYT$, form the lowest row of $S$ by first taking the smallest entry, say $i$, in the first column of $V$ and placing it in row $i$. Then fill out this row by choosing the largest entry from column $2$ of $V$ that is weakly smaller than $i$, the largest entry from column $3$ of $V$ that is weakly smaller than the entry chosen for column $2$, etc. Delete all chosen entries from $V$, then repeat the algorithm to make the second-lowest row of $S$, etc. See Figure~\ref{fig:leftright} for an example.

\begin{lemma}
The left row-filling algorithm is a well-defined decomposition of a reverse tableau $V$, i.e., all runs use one entry in the first column, and all entries of the reverse tableau appear in some run.
\end{lemma}
\begin{proof}
The statement is clear if $V$ has only one row. We show that deleting the boxes from such a run yields another $\revSSYT$; the statement then follows by induction on the number of rows of $V$. Consider a pair of adjacent columns $c$, $c+1$ of $V$. Assuming a run uses an entry of both column $c$ and $c+1$ (i.e., the run doesn't end in column $c$), first observe that the entry $j$ of column $c+1$ chosen by the run must be weakly above the entry $i$ of column $c$ chosen by the run, since otherwise the entry immediately right of $i$ is strictly greater than $j$ but less than or equal to $i$, so this entry would have been chosen instead. When boxes $i$ and $j$ are removed, for all rows $r$ between the row of $j$ and the row immediately above $i$, the box of row $r$ in column $c$ is now adjacent to the box of row $r+1$ in column $c+1$. But this preserves weakly decreasing rows, since entries of $V$ decrease down columns. If the run uses an entry of column $c$ is but not $c+1$, then all boxes in column $c+1$ have strictly larger entry than the removed box in column $c$, hence this removed box has no entry immediately right of it. Hence removing such a run yields a new $\revSSYT$.
\end{proof}

\begin{lemma}
The left row-filling algorithm yields the same $\ASSF$ as the column-filling algorithm.
\end{lemma}
\begin{proof}
Suppose for a contradiction that the two algorithms disagree. Consider the lowest then leftmost box $C$ (i.e. earliest in terms of the left row-filling) where the entry $i$ given by the left row-filling differs from the corresponding entry $j$ given by the column-filling. 

First suppose that $i<j$. The entry of the box immediately left of $C$ is the same in both fillings, and the column filling was able to place $j$ in $C$. The left row-filling, which places the largest possible entry available, placed $i<j$ in $C$, so the only way this could have happened is if the left row-filling already placed $j$ in a lower row. This contradicts the assumption that $C$ is the lowest then leftmost box at which the algorithms differ. Note the same contradiction occurs if the column-filling places $j$ in $C$ and the left row-filling places nothing in $C$.

Now suppose that $i>j$. The entry of the box immediately left of $C$ is the same in both fillings, and the left row-filling was able to place $i$ in $C$. The column-filling places entries of a column from largest to smallest, placing each as low as possible. Thus the column filling placed $i$ in the column of $C$ before it placed $j$ in $C$, so when it placed $i$, $C$ was available. So since the column-filling did not place $i$ in $C$, it must have placed $i$ lower in the column. This again contradicts the assumption that $C$ is the lowest then leftmost box at which the algorithms differ. Note the same contradiction occurs if the left row-filling places $i$ in $C$ and the column-filling places nothing in $C$.
\end{proof}

We now establish a column-set preserving map $\phi:\qKT^{(1)}\rightarrow \revSSYT$. 
Given $T\in \qKT^{(1)}$, let $\phi(T)$ be the tableau created by top-justifying $T$ and reordering entries in each column so they decrease from top to bottom. 

\begin{lemma}
If $T\in \qKT^{(1)}$, then $\phi(T)\in \revSSYT$, and $\phi(T)$ preserves column sets.
\end{lemma}
\begin{proof}
It is clear that column sets are preserved. Since $T$ is a filling of a weak composition diagram, the number of entries in column $c$ of $T$ is weakly larger than the number of entries in column $c+1$ of $T$, for all $c$. Thus $\phi(T)$ has partition shape. Since $T$ does not repeat entries in any column, by construction the columns of $\phi(T)$ are strictly decreasing. Since every entry in column $c+1$ appears immediately right of a weakly larger entry in column $c$, the rows of $\phi(T)$ are weakly decreasing. Thus $\phi(T)\in \revSSYT$.
\end{proof}

We want to show $\phi$ is a bijection. To this end, we define a map $\psi: \revSSYT\rightarrow \qKT^{(1)}$ as follows. Given $V\in \revSSYT$, 
we build a skyline filling $T$ by rows. Starting with the rightmost column set, say $C_k$, of $V$, choose the smallest number in $C_k$. Then choose the smallest number in $C_{k-1}$ that is weakly larger than $k$. Continue in this manner until you end by choosing a number $\ell$ in $C_1$. The elements chosen form a row of $T$ with row index $\ell$. Then delete all chosen elements from $V$ and repeat.  Call this the \emph{right row-filling algorithm}: see Figure~\ref{fig:leftright} for an example.

\begin{lemma}
Right row-filling is a well-defined decomposition of a reverse tableau, i.e., all runs use one entry in the first column, and all entries of the reverse tableau appear in some run.
\end{lemma}
\begin{proof}
Like for left row-filling, the statement is clear if $V$ has only one row. We show that deleting the boxes from such a run yields another $\revSSYT$; the statement follows by induction on the number of rows of $V$. Consider a pair of adjacent columns $c$, $c+1$ of $V$. Assuming a run uses an entry of both column $c$ and $c+1$ (i.e., the run doesn't start in column $c$), first observe that the entry $j$ of column $c+1$ chosen by the run must be weakly above the entry $i$ of column $c$ chosen by the run, since otherwise the entry to the left of $j$ is greater than or equal to $j$ but strictly smaller than $i$, so this entry would have been chosen instead. When boxes $i$ and $j$ are removed, for all rows $r$ between the row of $j$ and the row immediately above $i$, the box of row $r$ in column $c$ is now adjacent to the box of row $r+1$ in column $c+1$. But this preserves weakly decreasing rows, since entries of $V$ decrease down columns. If the run uses an entry of column $c$ but not $c+1$, then by definition column $c+1$ is empty. Hence removing such a run yields a new $\revSSYT$. 

Since $V\in\revSSYT$, the run found by this process always ends in the first column: for any box you choose in column $c+1$ there is a box with a greater than or equal value in column $c$. Hence by induction, right row-filling gives a well-defined decomposition of a $\revSSYT$.
\end{proof}

\begin{figure}[ht]
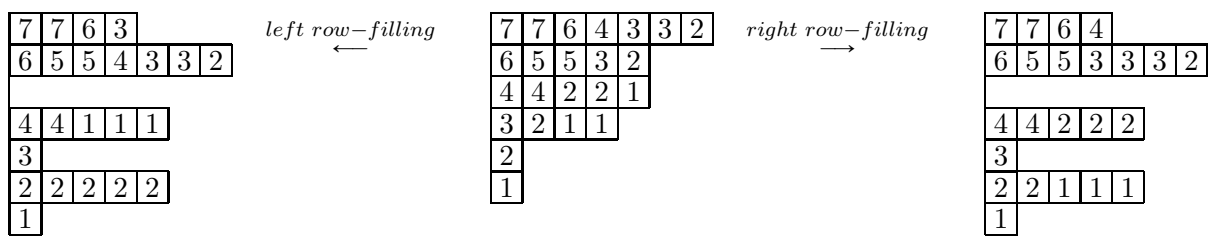

  \begin{center}
    \begin{displaymath}
      \begin{array}{c@{\hskip\cellsize}c@{\hskip\cellsize}c}
        \vline\tableau{ 7 & 7 & 6 & 3 \\ 6 & 5 & 5 & 4 & 3 & 3 & 2 \\ \\  4 & 4 & 1 & 1 & 1 \\ 3 \\ 2 & 2 & 2 & 2 & 2 \\ 1 } \,\,\,\,\,\,\, \substack{left \,\, row-filling \\ \longleftarrow} \,\,\,\,\, &
        \tableau{ 7 & 7 & 6 & 4 & 3 & 3 & 2  \\  6 & 5 & 5 & 3 & 2 \\  4 & 4 & 2 & 2 & 1   \\  3 & 2 & 1 & 1 \\ 2 \\ 1  } \,\,\,\,\,\,\, \substack{right \,\, row-filling \\ \longrightarrow} \,\,\,\,\, &
        \vline\tableau{ 7 & 7 & 6 & 4 \\ 6 & 5 & 5 & 3 & 3 & 3 & 2 \\ \\ 4 & 4 & 2 & 2 & 2 \\ 3 \\  2 & 2 & 1 & 1 & 1 \\ 1 } 
      \end{array}
    \end{displaymath}
    \caption{\label{fig:leftright}The left and right row-filling algorithms on a reverse tableau.}
  \end{center}
\end{figure}

\begin{lemma}\label{lem:SSYTtoqKT}
For any $V\in \revSSYT$, $\psi(V)\in \qKT^{(1)}$.
\end{lemma}
\begin{proof}
The right row-filling algorithm is straightforwardly equivalent to the \emph{thread map} \cite[Definition 3.5, Lemma 3.6]{AS2} on Kohnert diagrams. Translating to quasi-key tableaux (see Remark~\ref{rmk:quasi-key}) the threads of \cite{AS2} are exactly the rows of the quasi-key tableaux. 
The statement then follows from \cite[Lemma 3.6]{AS2} and the fact that for every $\revSSYT$ of shape $\lambda$, the entries index the positions of the boxes in a Kohnert diagram for $D(\lambda)$.
\end{proof}

\begin{lemma}\label{lem:colsets}
No two elements of $\qKT^{(1)}$ have identical column sets.
\end{lemma}
\begin{proof}
Suppose two elements $T, T'\in \qKT^{(1)}$ had identical column sets. Then $\phi(T) = \phi(T')=V$, where $V$ is the (unique) reverse tableau with the same column sets as $T, T'$. By Lemma~\ref{lem:SSYTtoqKT}, the right row-filling algorithm on $V$ yields an element of $\qKT^{(1)}$, so we may assume without loss of generality that $\psi(V)=T$. To create the rows of $T'$, we can also decompose $V$ into decreasing runs by starting at the rightmost column set and moving leftward (as in the right row-filling algorithm), except we have freedom over which entry we choose from the next column to the left at each step (up to making sure we choose entries in such a way that the decomposition can be completed): we are not restricted to choosing the smallest available entry. 

Since $T\neq T'$, at some point when constructing $T'$ in this way, we deviate from $\psi$ by choosing an entry larger than the smallest possible entry. Suppose the first time a choice of entry $i$ for $T'$ that is larger than the entry $j$ chosen for $T$ by the right row-filling algorithm occurs in column $c$. Let $R$ be the row of $T'$ that is being created at this moment. Since $R$ avoids the entry $j$, a row, say $S$, created later than $R$ must use $j$. Since lengths of rows weakly decrease throughout this process, we have row $S$ is weakly shorter than row $R$, and in column $c$, the entry of $S$ is smaller than the entry of $R$. Moreover, we claim that in all columns to the right of $c$, the entry of $S$ is smaller than the entry of $R$ (assuming both exist). To see this, suppose that $S$ had a larger entry than $R$ in some column after $c$. Since row $R$ is weakly longer than row $S$, one can check this forces a violation of quasi-key tableau condition (3), regardless of whether row $R$ is above or below row $S$ in $T'$.

We may assume that when creating rows of $T'$, if we ever \emph{start} a row in column $c$, then we may without loss of generality choose to first create the row starting with the smallest available entry in column $c$. This is because if we reach a point where we're starting a row in column $c$, then every remaining entry of column $c$ also starts a row (in column $c$). Since, when creating $T'$, we have freedom on how we choose the remaining entries for any of these rows, it does not matter what order we declare these rows are created in.

Therefore, we may assume that at the first point we deviate from $\psi$, the row $R$ of $T'$ that we are constructing did not start in this column $c$. Consider the value $k$ chosen for $R$ in column $c+1$. We must have $k\le j$, since we could have chosen $j$ for $R$ in column $c$. As before, let row $S$ of $T'$ be the row that chooses entry $j$ in column $c$. This means row $R$ has entries $i$ and $k$ in columns $c$ and $c+1$ respectively, and row $S$ has entry $j$ in column $c$, where $i>j$ and $k\le j$. If row $R$ is below row $S$ in $T'$, this configuration violates quasi-key tableau condition (3). So assume row $R$ is above row $S$. Now if row $R$ is strictly longer than row $S$, the same configuration violates quasi-key tableau condition (4). Since we already established row $R$ is weakly longer than row $S$, this means rows $R$ and $S$ must have the same length.

In particular, since row $R$ has an entry $k$ in column $c+1$, row $S$ must also have an entry, say $\ell$, in column $c+1$. We have established that in all columns to the right of $c$, the entry of $S$ is smaller than the entry of $R$, hence $\ell<k$. But choosing the column $c$ entry for row $R$ was the first time we deviated from choosing the smallest entry available. So when choosing the column $c+1$ entry for row $R$ we did not choose $\ell$ (even though $\ell$ was not yet chosen at this point, since $\ell$ was chosen later by $S$). The only way this could have happened was for row $R$ to have chosen an element $p>\ell$ of column $c+2$ . Then since rows $R$ and $S$ have the same length, $S$ also has an entry $q<p$ in column $c+2$.  But by the same argument, since we did not choose $q$ for for row $R$, there must be an element $u>q$ in column $c+3$  chosen by $R$, etc. In particular, the existence of an entry of row $S$ in some column $d$ implies the existence of an entry of row $R$ in column $d+1$, contradicting that rows $R$ and $S$ have the same length.
\end{proof}

\begin{theorem}\label{thm:bijection1}
The map $\psi: \revSSYT \rightarrow \qKT^{(1)}$ is a bijection, with inverse map $\phi$.
\end{theorem}
\begin{proof}
Elements of both $\revSSYT$ and (by Lemma~\ref{lem:colsets}) $\qKT^{(1)}$ have distinct column sets. Since both $\phi$ and $\psi$ preserve column sets, these maps are mutually inverse bijections.
\end{proof}

\begin{remark}
Notice the duality in the construction of $\ASSF$ and $\qKT^{(1)}$ from $\revSSYT$. The $\ASSF$ are constructed by taking minimally decreasing runs from left to right, the $\qKT^{(1)}$ by taking minimally increasing runs from right to left.
\end{remark}

\subsection{Restrictions defined by fundamental particles and fundamental slide polynomials}

We now establish restrictions of this bijection to some distinguished subsets of $\qKT^{(1)}$ and $\ASSF$. The first case arises from the combinatorics of fundamental particles, and the second from that of fundamental slide polynomials.

Say that $V\in \revSSYT$ (respectively, $T\in \qKT^{(1)}$) is \emph{particle-highest} if for every entry $i$ appearing in $V$ (respectively, $T$), the leftmost $i$ is either in the first column of $V$ (respectively, $T$), or is weakly left of some entry $j$, where $j$ is the smallest entry larger than $i$ appearing in $V$ (respectively, $T$). Let $\HrevSSYT$ (respectively, $\HqKT^{(1)}$) be the set of particle-highest $\revSSYT$ (respectively, $\qKT^{(1)}$).

By an easy check similar to that of Lemma~\ref{lem:dst}, the destandardization map for $\ASSF$ introduced in the previous section is well defined on both $\revSSYT$ and $\qKT^{(1)}$, and for $V\in \revSSYT$ (respectively, $T\in \qKT^{(1)}$), we have $\destand(V)=V$ (respectively, $\destand(T)=T$) if and only if $V$ (respectively, $T$) is particle-highest. 

\begin{theorem}\label{thm:Hbijection}
The bijection $\qKT^{(1)}\leftrightarrow \revSSYT \leftrightarrow \ASSF$ restricts to a bijection $\HqKT^{(1)} \leftrightarrow \HrevSSYT \leftrightarrow \HSSF$.
\end{theorem}
\begin{proof}
The condition under which destandardization applies is a condition only on column sets, which are preserved under the bijections, and we are using the same destandardization map on $\qKT^{(1)}$, $\revSSYT$ and $\ASSF$, hence $\destand$ changes $i$'s to $i+1$'s in $T\in \qKT^{(1)}$ if and only if it changes $i$'s to $i+1$'s in the corresponding $V\in \revSSYT$ if and only if it changes $i$'s to $i+1$'s in the corresponding $S\in \ASSF$.  Moreover, $\destand$ changes \emph{all} $i$'s to $i+1$'s in each of $T$, $V$ and $S$, so the column sets of $\destand(T)$, $\destand(V)$ and $\destand(S)$ are identical, and hence these three objects are identified under the bijection. In particular, $\destand$ commutes with the bijection.
\end{proof}

We can restrict the bijection further. Say that a $\qKT^{(1)}$, $\revSSYT$ or $\ASSF$ $X$ is \emph{quasi-Yamanouchi} if for every $i$ appearing in $X$, the leftmost $i$ is either in the first column of $X$, or weakly left of some $i+1$ in $X$. Let $\QqKT^{(1)}$, $\QrevSSYT$, $\QSSF$ denote the subset of quasi-Yamanouchi elements of, respectively, $\qKT^{(1)}$, $\revSSYT$, $\ASSF$.

\begin{remark} 
This definition of quasi-Yamanouchi is actually the same as that in Definition~\ref{def:quasi-Yamanouchi}, restricted to $\qKT^{(1)}$. This is because an entry $i$ being in the first column of $T \in \qKT^{(1)}$ implies this entry $i$ is in row $i$. However, this definition of quasi-Yamanouchi for $\revSSYT$ is slightly different to the definition of quasi-Yamanouchi for $\revSSYT$ in \cite{AS1}, which puts no condition on the first column.
\end{remark}

Notice the quasi-Yamanouchi condition is stronger than the particle-highest condition, in particular, quasi-Yamanouchi objects are a subset of particle-highest objects.

We define another destandardization map $\destand_Q$ on $\qKT^{(1)}$, $\revSSYT$, $\ASSF$: $\destand_Q$ changes all $i$'s to $i+1$'s if the leftmost $i$ is not in the first column and has no $i+1$ weakly to its right. By replacing $j$ with $i+1$ in the argument of Lemma~\ref{lem:dst}, $\destand_Q$ is well-defined and $\destand_Q(X) = X$ if and only if $X$ is quasi-Yamanouchi (for $X$ a $\qKT^{(1)}$, $\revSSYT$, $\ASSF$).

\begin{theorem}\label{thm:Qbijection}
The bijection $\qKT^{(1)}\leftrightarrow \revSSYT \leftrightarrow \ASSF$ restricts to a bijection $\QqKT^{(1)} \leftrightarrow \QrevSSYT \leftrightarrow \QSSF$.
\end{theorem}
\begin{proof}
Essentially identical to the proof of Theorem~\ref{thm:Hbijection}.
\end{proof}

\subsection{Restriction to fixed skyline diagrams}

Define a map $\Psi: \qKT^{(1)}\rightarrow \ASSF$ by letting $\Psi(T)$ be the $\ASSF$ obtained by performing left row-filling on $\phi(T)$. For example, $\Psi$ maps the column quasi-key tableau on the right hand side of Figure~\ref{fig:leftright} to the semi-skyline filling on the left hand side of Figure~\ref{fig:leftright}. Since $\phi$ and left row-filling are bijections, $\Psi$ is also a bijection. Moreover, since $\phi$ preserves column sets and left row-filling acts on sets of entries in a column without regard to their row index, $\Psi$ may be interpreted simply as performing left row-filling on the column sets of a column quasi-key tableau, without reference to $\phi$ or reverse tableaux. 

As in the proof of Lemma~\ref{lem:lswap}, for $a$ a weak composition we define $\Shuffle(a)$ to be the set of all weak compositions whose entries are a rearrangement of the entries of $a$. The following is immediate from the definition of left row-filling:

\begin{lemma}\label{lem:shuffle}
Let $a$ be a weak composition. If $T\in \qKT^{(1)}(a)$, then $\Psi(T) \in \bigcup_{b\in \Shuffle(a)}\ASSF(b)$.
\end{lemma}

Even though it is suggestive from Figure~\ref{fig:leftright}, it is not clear from earlier results of this section that $\Psi$ in fact preserves the shape of a skyline diagram: if $T\in \qKT^{(1)}(a)$ then $\Psi(T) \in \ASSF(a)$. The remainder of this section is devoted to proving this. 

A family of fillings of skyline diagrams called \emph{(semistandard) key tableaux} are defined in \cite{Assaf-nonsymmetric}. Key tableaux are a reinterpretation of the \emph{Kohnert tableaux} of \cite{AS2}, in the exact same way that quasi-key tableaux reinterpret the quasi-Kohnert tableaux of \cite{AS2}; see Remark~\ref{rmk:quasi-key}. We will not require the definition of key tableaux here, only the fact that the quasi-key tableaux $\qKT(a)$ (and thus the column quasi-key tableaux $\qKT^{(1)}(a)$) are a subset of the set $\KeyTab(a)$ of key tableaux for $a$; this follows straightforwardly from the fact \cite{AS2} that quasi-Kohnert tableaux are a subset of Kohnert tableaux. The following is proved in \cite{MPS18}:

\begin{theorem}\label{thm:MPS}\cite{MPS18}
Let $a$ be a weak composition. Then left row-filling applied to the column sets of a key tableau defines a bijection
\[\KeyTab(a) \rightarrow \bigcup_{b\in \lswap(a)}\ASSF(b).\]
\end{theorem}

In particular, Theorem~\ref{thm:MPS} tightens Lemma~\ref{lem:shuffle}: since $\qKT^{(1)}(a)\subset \KeyTab(a)$ it follows that if $T\in \qKT^{(1)}(a)$ then $\Psi(T) \in \bigcup_{b\in \lswap(a)}\ASSF(b)$. This result allows us to prove that $\Psi$ preserves skyline diagrams. 

\begin{theorem}\label{thm:restricted}
Let $a$ be a weak composition. The bijection $\qKT^{(1)}\leftrightarrow \ASSF$ defined by $\Psi$ restricts to a bijection $\qKT^{(1)}(a) \leftrightarrow \ASSF(a)$.
\end{theorem}
\begin{proof}
Since $\Psi$ is a bijection, and since $\qKT^{(1)}(a)$ and $\ASSF(a)$ are equinumerous and finite for any $a$ (by Theorem~\ref{thm:colquasikeyisatom} they both generate the Demazure atom $\atom_a$), it is enough to prove that if $T\in \qKT^{(1)}(a)$ then $\Psi(T)\in \ASSF(a)$. 
Since $\qKT^{(1)}(a)\subset \KeyTab(a)$, if $T\in \qKT^{(1)}(a)$ then by Theorem~\ref{thm:MPS} we have $\Psi(T)\in \ASSF(b)$ for some $b\in \lswap(a)$. 

Define a poset structure on $\Shuffle(a)$ by setting $b\le b'$ if $b\in \lswap(b')$. It follows from the definition of $\lswap$ (Definition~\ref{def:lswap}) that this relation defines a poset, whose maximal element is the weak composition $c$ consisting of the entries of $a$ in increasing order and whose minimal element is the weak composition $d$ consisting of the entries of $a$ in decreasing order.

We now induct on this poset. Since $d$ is minimal in this poset, i.e., $\lswap(d) = \{d\}$, it follows from Theorem~\ref{thm:MPS} that if $T\in \qKT^{(1)}(d)$ then $\Psi(T)\in \ASSF(d)$. Therefore $\Psi$ is a bijection from $\qKT^{(1)}(d)$ to $\ASSF(d)$. Now consider $a$ and suppose inductively that for every $b<a$, if $T\in \qKT^{(1)}(b)$ then $\Psi(T) \in \ASSF(b)$. Suppose $T\in \qKT^{(1)}(a)$. By Theorem~\ref{thm:MPS}, we have $\Psi(T)\in \bigcup_{b\le a}\ASSF(b)$. But by the inductive hypothesis, if $S\in \ASSF(b)$ for some $b<a$, then $\Psi^{-1}(S)\in \qKT^{(1)}(b)$. Therefore we must have $\Psi(T)\in \ASSF(a)$, as required. 
\end{proof}

Combining Theorem~\ref{thm:restricted} and Theorems~\ref{thm:Hbijection} and \ref{thm:Qbijection}, we obtain

\begin{corollary}
The bijection $\HqKT^{(1)} \leftrightarrow \HSSF$ from Theorem~\ref{thm:Hbijection} restricts to a bijection $\HqKT^{(1)}(a) \leftrightarrow \HSSF(a)$, and the bijection $\QqKT^{(1)} \leftrightarrow \QSSF$ from Theorem~\ref{thm:Qbijection} restricts to a bijection $\QqKT^{(1)}(a) \leftrightarrow \QSSF(a)$.
\end{corollary}

\section*{Acknowledgements}
The author is grateful to J. Haglund, S. Mason and S. van Willigenburg for illuminating discussion, and to an anonymous referee whose very helpful suggestions and comments improved the exposition of this paper.

%
%

\bibliographystyle{amsalpha} 
\bibliography{Bases.bib}

\providecommand{\bysame}{\leavevmode\hbox to3em{\hrulefill}\thinspace}
\providecommand{\MR}{\relax\ifhmode\unskip\space\fi MR }
\providecommand{\MRhref}[2]{%
  \href{http://www.ams.org/mathscinet-getitem?mr=#1}{#2}
}
\providecommand{\href}[2]{#2}
\begin{thebibliography}{HLMvW11b}

\bibitem[Ale16]{Ale16}
Per Alexandersson, \emph{Non-symmetric {M}acdonald polynomials and
  {D}emazure-{L}usztig operators}, preprint (2016), {\sf arXiv:1602.05153}.

\bibitem[AS17]{AS1}
Sami Assaf and Dominic Searles, \emph{Schubert polynomials, slide polynomials,
  {S}tanley symmetric functions and quasi-{Y}amanouchi pipe dreams}, Adv. Math.
  \textbf{306} (2017), 89--122.

\bibitem[AS18]{AS2}
\bysame, \emph{{K}ohnert tableaux and a lifting of quasi-{S}chur functions}, J.
  Combin. Theory Ser. A \textbf{156} (2018), 85--118.

\bibitem[Ass17a]{Assaf-models}
Sami Assaf, \emph{Combinatorial models for {S}chubert polynomials}, preprint
  (2017), {\sf arXiv:1703.00088}.

\bibitem[Ass17b]{Assaf-weak}
\bysame, \emph{Weak dual equivalence for polynomials}, preprint (2017), {\sf
  arXiv:1702.04051}.

\bibitem[Ass17c]{Assaf-nonsymmetric}
\bysame, \emph{Nonsymmetric {M}acdonald polynomials and a refinement of
  {K}ostka--{F}oulkes polynomials}, Trans. Amer. Math. Soc. (to appear,
  accepted 2017).

\bibitem[BJS93]{BJS93}
Sara~C. Billey, William Jockusch, and Richard~P. Stanley, \emph{Some
  combinatorial properties of {S}chubert polynomials}, J. Algebraic Combin.
  \textbf{2} (1993), no.~4, 345--374. \MR{1241505 (94m:05197)}

\bibitem[Che95]{Cherednik}
Ivan Cherednik, \emph{Nonsymmetric {M}acdonald polynomials}, Int. Math. Res.
  Not. \textbf{1995} (1995), no.~10, 483--515.

\bibitem[Dem74]{Dem74}
Michel Demazure, \emph{Une nouvelle formule des caract\`eres}, Bull. Sci. Math.
  (2) \textbf{98} (1974), no.~3, 163--172. \MR{0430001 (55 \#3009)}

\bibitem[Ges84]{Ges84}
Ira~M. Gessel, \emph{Multipartite {$P$}-partitions and inner products of skew
  {S}chur functions}, Combinatorics and algebra (Boulder, Colo., 1983),
  Contemp. Math., vol.~34, Amer. Math. Soc., Providence, RI, 1984,
  pp.~289--317.

\bibitem[HHL08]{HHL08}
James Haglund, Mark Haiman, and Nick Loehr, \emph{A combinatorial formula for
  nonsymmetric {M}acdonald polynomials}, Amer. J. Math. \textbf{130} (2008),
  no.~2, 359--383.

\bibitem[HLMvW11a]{HLMvW11a}
J.~Haglund, K.~Luoto, S.~Mason, and S.~van Willigenburg, \emph{Quasisymmetric
  {S}chur functions}, J. Combin. Theory Ser. A \textbf{118} (2011), no.~2,
  463--490. \MR{2739497}

\bibitem[HLMvW11b]{HLMvW11b}
\bysame, \emph{Refinements of the {L}ittlewood-{R}ichardson rule}, Trans. Amer.
  Math. Soc. \textbf{363} (2011), no.~3, 1665--1686. \MR{2737282}

\bibitem[Ion03]{Ion}
Bogdan Ion, \emph{Nonsymmetric {M}acdonald polynomials and {D}emazure
  characters}, Duke Math. J. \textbf{116} (2003), no.~2, 299--318.

\bibitem[Koh91]{Koh91}
Axel Kohnert, \emph{Weintrauben, {P}olynome, {T}ableaux}, Bayreuth. Math. Schr.
  (1991), no.~38, 1--97, Dissertation, Universit{\"a}t Bayreuth, Bayreuth,
  1990. \MR{1132534}

\bibitem[KY04]{KY04}
Allen Knutson and Alexander Yong, \emph{A formula for {K}-theory truncation
  {S}chubert calculus}, Int. Math. Res. Not. \textbf{2004} (2004), no.~70,
  3741--3756.

\bibitem[LP07]{LP07}
Thomas Lam and Pavlo Pylyavskyy, \emph{Combinatorial {H}opf algebras and
  {K}-homology of {G}rassmanians}, Int. Math. Res. Not. \textbf{2007} (2007),
  rnm125.

\bibitem[LS82]{LS82}
Alain Lascoux and Marcel-Paul Sch{\"u}tzenberger, \emph{Polyn\^omes de
  {S}chubert}, C. R. Acad. Sci. Paris S\'er. I Math. \textbf{294} (1982),
  no.~13, 447--450. \MR{660739 (83e:14039)}

\bibitem[LS90]{LS90}
\bysame, \emph{Keys \& standard bases}, Invariant theory and tableaux
  ({M}inneapolis, {MN}, 1988), IMA Vol. Math. Appl., vol.~19, Springer, New
  York, 1990, pp.~125--144. \MR{1035493 (91c:05198)}

\bibitem[Mac91]{Mac91}
I.~G. Macdonald, \emph{Notes on {S}chubert polynomials}, LACIM, Univ. Quebec a
  Montreal, Montreal, PQ, 1991.

\bibitem[Mac96]{Macdonald}
\bysame, \emph{Affine {H}ecke algebras and orthogonal polynomials},
  Ast\'erisque \textbf{237} (1996), no.~4, 189--207, S\'eminare Bourbaki
  1994/95, Exp. no. 797.

\bibitem[Mas08]{Mason08}
Sarah Mason, \emph{A decomposition of {S}chur functions and an analogue of the
  {R}obinson-{S}chensted-{K}nuth algorithm}, S{\'e}m. Lothar. Combin.
  \textbf{57} (2008), B57e.

\bibitem[Mas09]{Mason}
\bysame, \emph{An explicit construction of type {A} {D}emazure atoms}, J.
  Algebraic Combin. \textbf{29} (2009), no.~3, 295--313.

\bibitem[Mon16]{Mon16}
Cara Monical, \emph{Set-{V}alued {S}kyline {F}illings}, preprint (2016), {\sf
  arXiv:1611.08777}.

\bibitem[MPS18]{MPS18}
Cara Monical, Oliver Pechenik, and Dominic Searles, \emph{Families of
  polynomials from combinatorial ${K}$-theory}, preprint (2018), {\sf
  arXiv:1806.03802}.

\bibitem[Opd95]{Opdam}
Eric~M Opdam, \emph{Harmonic analysis for certain representations of graded
  {H}ecke algebras}, Acta Math. \textbf{175} (1995), no.~1, 75--121.

\bibitem[PS17]{PS17}
Oliver Pechenik and Dominic Searles, \emph{Decompositions of {G}rothendieck
  polynomials}, Int. Math. Res. Not., to appear (2017), 28 pages, {\sf
  arXiv:1611.02545}.

\bibitem[Pun16]{Pun16}
Anna Pun, \emph{On decomposition of the {P}roduct of {D}emazure {A}toms and
  {D}emazure {C}haracters}, preprint (2016), {\sf arXiv:1606.02291}.

\bibitem[RS95]{RS95}
Victor Reiner and Mark Shimozono, \emph{Key polynomials and a flagged
  {L}ittlewood-{R}ichardson rule}, J. Combin. Theory Ser. A \textbf{70} (1995),
  no.~1, 107--143. \MR{1324004}

\bibitem[TvW15]{TvW15}
Vasu~V Tewari and Stephanie~J van Willigenburg, \emph{Modules of the 0-{H}ecke
  algebra and quasisymmetric {S}chur functions}, Adv. Math. \textbf{285}
  (2015), 1025--1065.

\end{thebibliography}

\end{document}